\documentclass[reqno,11pt]{amsart}
\usepackage{amsmath}
\usepackage{amssymb}
\usepackage{tabularx}
\usepackage{enumerate}
\usepackage[dvips]{psfrag,graphicx}
\usepackage{color}
\usepackage{subfigure}
\usepackage{cases}

\usepackage{mathrsfs}

\usepackage{amsfonts,amssymb,amsmath}
\usepackage{epsfig}

\usepackage{comment}

\makeatletter
\@addtoreset{equation}{section}
  
\makeatother
\newtheorem{thm}{Theorem}[section]
\newtheorem{lem}[thm]{Lemma}   
\newtheorem{cor}[thm]{Corollary}
\newtheorem{prop}[thm]{Proposition}
\newtheorem{rem}[thm]{Remark}

\newcommand{\sbt}{\,\begin{picture}(-1,1)(-1,-3)\circle*{3}\end{picture}\ }
\begin{document}
\title[Full cross-diffusion limit in the SKT model]
{Global structure of steady-states to \\the full cross-diffusion limit
in the Shigesada-Kawasaki-Teramoto model}
 \thanks{This research was
partially supported by JSPS KAKENHI Grant Number 19K03581.}
\author[K. Kuto]{Kousuke Kuto$^\dag$}
\thanks{$\dag$ Department of Applied Mathematics, 
Waseda University, 
3-4-1 Ohkubo, Shinjuku-ku, Tokyo 169-8555, Japan.}
\thanks{{\bf E-mail:} \texttt{kuto@waseda.jp}}
\date{\today}

\begin{abstract} 
In a previous paper \cite{Ku2}, 
the author studied the asymptotic behavior of 
coexistence steady-states to the Shigesada-Kawasaki-Teramoto
model as both cross-diffusion coefficients tend to infinity
at the same rate.
As a result, he proved that the asymptotic behavior can be
characterized by a limiting system that consists of
a semilinear elliptic equation and an integral constraint.
This paper studies the set of solutions of the limiting system.
The first main result gives sufficient conditions for
the existence/nonexistence of nonconstant solutions
to the limiting system by a topological approach using
the Leray-Schauder degree.
The second main result exhibits a bifurcation diagram
of nonconstant solutions to the one-dimensional limiting system by
analysis of a weighted time-map and a nonlocal constraint.
\end{abstract}

\subjclass[2020]{35B09, 35B32, 35B45, 35A16, 35J25, 92D25}
\keywords{cross-diffusion,
competition model,
limiting system, 
nonlinear elliptic equations,
integral constraint,
the Leray-Schauder degree,
bifurcation} \maketitle

\section{Introduction}
In this paper, we are concerned with
the following Neumann problem of
quasilinear elliptic equations:
\begin{equation}\label{SKT}
\begin{cases}
\Delta [\,(d_{1}+\alpha v)u\,]+
f(u,v)=0
\ \ &\mbox{in}\ \Omega,\\
\Delta [\,(d_{2}+\beta u)v\,]+
g(u,v)=0
\ \ &\mbox{in}\ \Omega,\\
u\ge 0,\ \ v\ge 0
\ \ &\mbox{in}\ \Omega,\\
\partial_{\nu}u=\partial_{\nu}v=0
\ \ &\mbox{on}\ \partial\Omega,
\end{cases}
\end{equation}
where
\begin{equation}\label{fgdef}
f(u,v):=u(a_{1}-b_{1}u-c_{1}v)
\quad\mbox{and}\quad
g(u,v):=v(a_{2}-b_{2}u-c_{2}v).
\end{equation}
The system \eqref{SKT} is derived from 
a diffusive Lotka-Volterra competition
model, where the unknown functions 
$u(x)$ and $v(x)$ represent 
the stationary population densities 
of the competing species in the habitat $\Omega$.
Throughout this paper,
$\Omega$ is assumed to be a bounded domain in $\mathbb{R}^{N}$ 
with a smooth boundary $\partial\Omega $ if $N\ge 2$;
an interval if $N=1$.
In \eqref{SKT},
$\Delta :=\sum^{N}_{j=1}\partial^{2}/\partial x_{j}^{2}$
denotes the Laplacian;
$\nu (x)$
denotes  the outward unit normal vector
at $x\in\partial\Omega $, and
$
\partial_{\nu }u=
\nu (x)\cdot\nabla u
$
is the out-flux of $u$.
In the reaction terms
$f(u,v)$ and $g(u,v)$,
coefficients
$a_{i}$,
$b_{i}$ and
$c_{i}$ 
$(i=1,2)$
are positive constants;
$a_{i}$ denote the birth rates of the respective species,
$b_{1}$ and $c_{2}$ denote the intra-specific competition coefficients;
$c_{1}$ and $b_{2}$ denote the inter-specific competition coefficients.
In the diffusion term pertaining to the Laplacian,
$d_{i}$
$(i=1,2)$
are positive constants;
$\alpha$ and $\beta $ are nonnegative constants,
$d_{1}\Delta u$ and $d_{2}\Delta v$ 
are linear diffusion terms describing a spatially random movement 
of each species.
whereas
$\Delta (uv)$ is a nonlinear diffusion term
describing an interaction of diffusion caused by 
the population pressure resulting from interference
between different species.
The interaction term $\Delta (uv)$ is often called
{\it cross-diffusion}
(see a book by Okubo and Levin \cite{OL} for 
modelings of the biological diffusion).
Such a Lotka-Volterra competition system with cross-diffusion 
(and some additional terms) 
was proposed by Shigesada, Kawasaki and Teramoto
\cite{SKT}.
Beyond their bio-mathematical aim to realize 
{\it segregation phenomena} of two competing species observed in 
ecosystems, a lot of pure mathematicians have studied a class of  
Lotka-Volterra systems with cross-diffusion as a prototype of 
diffusive interactions.
Such a class of Lotka-Volterra systems with cross-diffusion 
is referred to as {\it the SKT model} celebrating the authors of
the pioneering paper 
\cite{SKT}.
We refer to book chapters by J\"ungel \cite{Ju}, Ni \cite{NiW},
and Yamada \cite{Yam1, Yam2}
as surveys for mathematical works relating to the SKT model.

Despite the long history of research on the SKT model,
there are few papers on the global structure of the set of
positive nonconstant solutions
(such as $u>0$ and $v>0$ in $\Omega$)
to \eqref{SKT} in case where $\alpha$ and $\beta$
are large.
The purpose of this paper is to study the asymptotic
behavior of positive nonconstant solutions of \eqref{SKT} 
in the full cross-diffusion limit 
as $\alpha$, $\beta\to\infty$
and $\alpha/\beta\to\gamma$ with some $\gamma >0$.

In the unilateral cross-diffusion limit as
$\alpha\to\infty$ with $\beta\ge 0$ fixed,
Lou and Ni \cite{LN2} established the $L^{\infty}(\Omega )$ 
a priori bound,
which is independent of $\alpha$,
for all solutions of \eqref{SKT}
in case $N\le 3$.
Furthermore, they proved that 
the asymptotic behavior of solutions as $\alpha\to\infty$
(with fixed $\beta\ge 0$) can be
characterized by a solution to either 
a limiting system of the first kind or 
a limiting system of the second kind.
We refer to 
\cite{KW, LNY, LNY2, MSY, NWX, WWX, Wu, WX}
and \cite{Ku1, LW1}
as papers on the limiting systems of the first
and second kinds,
respectively.  

In the full cross-diffusion limit as 
$\alpha$, $\beta\to\infty$ and 
$\alpha/\beta\to\gamma >0$,
the author obtained 
the following $L^{\infty}(\Omega )$ a priori bound
for all solutions of \eqref{SKT} without any restriction on $N$:
\begin{thm}[\cite{Ku2}]\label{aprthm}
For any small $\eta >0$, 
there exists a positive constant $C=C(\eta, d_{i},a_{i}, b_{i}, c_{i})$
such that
if $\alpha>0$ and $\beta >0$
satisfy $\eta\le\alpha/\beta\le 1/\eta$,
then any solution $(u,v)$ of \eqref{SKT}
satisfies 
$$\max_{x\in\overline{\Omega}}u(x)\le C
\quad\mbox{and}\quad
\max_{x\in\overline{\Omega}}v(x)\le C.$$
\end{thm}
By a combination of Theorem \ref{aprthm} and
the elliptic regularity theory, the author obtained
the following limiting systems which can
characterize the asymptotic behavior of solutions of 
\eqref{SKT} in the full cross-diffusion limit.
This result gives a rigorous justification of 
a formal observation by Kan-on \cite{Ka}
on the existence of the limiting systems.
\begin{thm}[\cite{Ku2}]\label{limthm}
Suppose that
$a_{1}/a_{2}\neq b_{1}/b_{2}$
and
$a_{1}/a_{2}\neq c_{1}/c_{2}$.
Let $\{(u_{n},v_{n})\}$ be any sequence of 
positive nonconstant solutions of \eqref{SKT} with
$\alpha=\alpha_{n}\to\infty$,
$\beta=\beta_{n}\to\infty$
and
$\gamma_{n}:=\alpha_{n}/\beta_{n}\to\gamma >0$
as $n\to\infty$.
Then either of the following two situations occurs,
passing to a subsequence if necessary;
\begin{enumerate}[{\rm (i)}]
\item
there exist a positive function $u\in C^{2}(\overline{\Omega })$ and
a positive number $\tau $ 
such that
$$\lim_{n\to\infty}(u_{n},v_{n})=\biggl(u,\dfrac{\tau}{u}\biggr)
\quad\mbox{in}\ C^{1}(\overline{\Omega})\times C^{1}(\overline{\Omega }),$$
and
\begin{equation}\label{deltadef}
w:=\delta u-\dfrac{\gamma \tau}{u}
\quad\mbox{with}\quad \delta:=\dfrac{d_{1}}{d_{2}}
\quad\mbox{and}\quad d:=d_{2}
\end{equation}
satisfies
a limiting system which consists of 
the semilinear elliptic equation
\begin{subequations}\label{IS}
\begin{align}\label{IS-1} 
d\Delta w &+f
\biggl(u,\dfrac{\tau}{u}\biggr)-\gamma g
\biggl(u,\dfrac{\tau}{u}\biggr)=0
\quad\mbox{in}\ \Omega,
\end{align}
subject the homogeneous Neumann boundary condition
\begin{equation}\label{IS-2}
\partial_{\nu}w=0\quad\mbox{on}\ \partial\Omega
\end{equation}
and the integral constraint
\begin{equation}\label{IS-3}
\displaystyle\int_{\Omega }
f\biggl(u,\dfrac{\tau}{u}\biggr)=0;
\end{equation}
\end{subequations}
\item
there exist nonnegative functions
$u$,
$v\in C(\overline{\Omega})$ such that
$uv=0$
in $\Omega$,
$$
\lim_{n\to\infty}(u_{n},v_{n})=(u,v)
\quad\mbox{uniformly in}\ \overline{\Omega}$$
and $w_{n}:=\delta u_{n}-\gamma \tau u_{n}^{-1}$ satisfies
$
\lim_{n\to\infty}w_{n}=w$
in
$C^{1}(\overline{\Omega })$
with some sign-changing function $w$
satisfying
\begin{equation}\label{CS}
\begin{cases}
d\Delta w+f\biggl(\dfrac{w_{+}}{\delta},\dfrac{w_{-}}{\gamma }\biggr)
-\gamma g\biggl(\dfrac{w_{+}}{\delta},\dfrac{w_{-}}{\gamma }\biggr)=0
\quad&\mbox{in}\ \Omega,\\
\partial_{\nu}w=0
\quad&\mbox{on}\ \partial\Omega,\\
\displaystyle\int_{\Omega }f\biggl(\dfrac{w_{+}}{\delta},
\dfrac{w_{-}}{\gamma }\biggr)=0
\end{cases}
\end{equation}
and
$$
(u,v)=\biggl(\dfrac{w_{+}}{\delta},
\dfrac{w_{-}}{\gamma }\biggr),
$$
where
$w_{+}:=\max\{w,0\}$ and
$w_{-}:=-\min\{w,0\}\ge 0$.
\end{enumerate}
\end{thm}
By Theorem \ref{limthm},
we know that a segregation of competing species can take place
in the sense that $u_{n}(x)v_{n}(x)$ converges to a nonnegative
constant $\tau $ in the full cross-diffusion limit as
$\alpha_{n}$,
$\beta_{n}\to\infty$
and 
$\gamma_{n}=\alpha_{n}/\beta_{n}\to\gamma >0$.
The second situation (ii) can be interpreted as
the {\it complete segregation} in which
territories of two competing species 
completely segregate each other because $\tau=0$,
while the first situation (i) can be interpreted as
the {\it incomplete segregation}
in which
territories of two competing species 
do not completely segregate because $\tau>0$.
In \cite{Ku2}, the author showed local bifurcation 
curves of nonconstant solutions of 
the limiting system \eqref{IS} of the incomplete segregation.
On the other hand, it was shown in \cite{Ku2} that
the complete segregation (ii) cannot occur in the 
one-dimensional case.

This paper first show that the complete segregation
(ii) cannot occur also in the higher dimensional case.
Then it becomes important to derive information on 
the set of nonconstant solutions of the limiting system 
\eqref{IS}
in order to know the segregation mechanism of
two competing species when
cross-diffusion coefficients
$\alpha$ and $\beta$ are sufficiently large. 
This paper will give some sufficient conditions
for the existence/nonexistence of nonconstant
solutions of the limiting system \eqref{IS}.
In what follows, we usually regard $(w,\tau )$
as a pair of unknowns with positive parameter
$d$.
It should be noted that
varying $d$ 
corresponds to varying $d_{1}$ and $d_{2}$ keeping 
$d_{1}/d_{2}=\delta$
(see \eqref{deltadef}),
and moreover,
$(u,\tau)$ also can be 
regarded as a pair of unknowns by 
the relation in \eqref{deltadef};
$$
w=\delta u-\dfrac{\gamma\tau}{u},
\quad\mbox{that is},\quad
u=\dfrac{\sqrt{w^{2}+4\gamma\delta\tau}+w}{2\delta}.
$$
Our strategy of the proof is as follows:
First we use the Poincar\'e inequality and the maximum principle
to show that 
the Neumann problem
\eqref{IS-1}-\eqref{IS-2} does not admit any
nonconstant solution 
if $d>0$ or $\tau>0$ is sufficiently large.
Then, in the weak competition case 
$c_{1}/c_{2}<a_{1}/a_{2}<b_{1}/b_{2}$
or the strong competition case 
$b_{1}/b_{2}<a_{1}/a_{2}<c_{1}/c_{2}$,
if $d>0$ is large enough, then \eqref{IS-1}-\eqref{IS-2}
admits a unique positive solution $(u,\tau )=(u^{*}, u^{*}v^{*})$,
such as $u>0$ in $\Omega$ and $\tau>0$,
where $(u^{*},v^{*})$ 
corresponds to the positive constant solution
\begin{equation}\label{const}
(u^{*},v^{*}):=
\dfrac{1}{b_{1}c_{2}-b_{2}c_{1}}
(a_{1}c_{2}-a_{2}c_{1}, b_{1}a_{2}-b_{2}a_{1})
\end{equation}
of \eqref{SKT}.
The index of an associated operator around
$(w^{*}, \tau^{*}):=
(\delta u^{*}-\gamma v^{*}, u^{*}v^{*})$
will be calculated, and especially 
for the case where $D(a_{i}, b_{i}, c_{i}, \gamma )>0$
(see \eqref{Ddef} for 
the definition of $D(a_{i}, b_{i}, c_{i}, \gamma )$),
the homotopy invariance property of the Leray-Schauder degree
will enables us to get sufficient intervals of $d$
for the existence of nonconstant solutions of \eqref{IS}.

In particular, 
for the one-dimensional case,
the global bifurcation structure of nonconstant
solutions of \eqref{IS} will be shown
(see Figure 1).
The procedure of the proof is as follows: 
First, the bifurcation structure of solutions to the Neumann problem
\eqref{IS-1}-\eqref{IS-2} is obtained through the analysis of the time-map 
of the related initial value problem. We note that the time map contains 
a weighted singular integral and differs from the usual one. 
Next, we construct a bifurcation branch of the set of solutions to 
the limit system \eqref{IS} by selecting functions in the set of 
solutions to \eqref{IS-1}-\eqref{IS-2} that satisfy 
the integral constraint \eqref{IS-3}. 
For the selection, 
we adopt a topological method combining 
the singular perturbation of solutions of \eqref{IS-1}-\eqref{IS-2}
and the global bifurcation theory.
As a result,
in the weak or strong competition case with
$D(a_{i}, b_{i}, c_{i}, \gamma )>0$,
for each $j\in\mathbb{N}$,
we shall construct a branch of solutions of \eqref{IS}
in which derivatives $u'$ change the sign exactly $j-1$ times 
bifurcating from a pitchfork bifurcation point on 
the branch of the constant solution $\{(d, u^{*}, \tau^{*})\,:\,d>0\}$
at some $d=d^{(j)}>0$,
and moreover, reaching a singular limit 
$(d, \tau)=(0, \tau_{0})$ with some $\tau_{0}\ge 0$.
Furthermore, it will be shown that
$\tau_{0}>0$ in the weak competition case, while
$\tau_{0}=0$ in the strong competition case.

The contents of this paper is as follows:
In Section 2, the nonexistence of nontrivial solution
of \eqref{CS} will be proved.
In Section 3, main results on the set of 
nonconstant solution of the limiting system \eqref{IS}
will be presented.
In Section 4, we derive some a priori estimates of 
solutions to \eqref{IS}.
In Section 5, we show sufficient conditions on 
the existence of nonconstant solutions of 
\eqref{IS} in the multi-dimensional case.
In Section 6, we construct the global bifurcation
curves of nonconstant solutions of \eqref{IS}
in the one-dimensional case.

Throughout this paper, 
the usual norms of the functional spaces $L^{p}(\Omega )$
for $p\in [\,1,\infty )$ and $L^{\infty }(\Omega )$
are denoted by
$$
\|u\|_{p}:=
\left(\displaystyle\int_{\Omega }|u(x)|^{p}\right)^{1/p}
\quad\mbox{and}\quad
\|u \|_{\infty }:=\mbox{ess.}\sup_{x\in\overline{\Omega }}|u(x)|.
$$
Hence $\|u\|_{\infty}=\max_{x\in\overline{\Omega}}|u(x)|$
in a case when $u\in C(\overline{\Omega })$.
Furthermore, we denote by $\{\varPhi_{j}\}^{\infty}_{j=0}$
a complete orthonormal base in $L^{2}(\Omega )$
consisting of eigenfunctions of 
$-\Delta $ with the homogeneous Neumann boundary condition
on $\partial\Omega$, namely,
\begin{equation}\label{Lap}
\begin{cases}
-\Delta\varPhi_{j}=\lambda_{j}\varPhi_{j}\quad &\mbox{in}\ \Omega,
\\
\|\varPhi_{j}\|_{2}=1,\\
\partial_{\nu}\varPhi_{j}=0
\quad &\mbox{on}\ \partial\Omega,
\end{cases}
\end{equation}
where 
$$
0=\lambda_{0}<\lambda_{1}\le\lambda_{2}\le\cdots\le
\lambda_{j}\le\lambda_{j+1}\le\cdots
$$
represent all eigenvalues counting multiplicity.

\section{Nonexistence of nonconstant solutions for
the complete segregation}
Throughout this paper,
the following maximum principle for elliptic equations
will play an important role in the proofs.

\begin{lem}[e.g., \cite{LN2}]\label{MPlem}
Suppose that $h\in C(\overline{\Omega}\times\mathbb{R})$.
Then the following properties {\rm (i)} and {\rm (ii)} hold true$:$
\begin{enumerate}[{\rm (i)}]
\item
If $\underline{U}\in C^{2}(\Omega )
\cap C^{1}(\overline{\Omega })$ satisfies
$$
\Delta \underline{U}+
h(x,\underline{U})\ge 0
\quad\mbox{in}\ \Omega,\quad
\partial_{\nu}\underline{U}\le 0
\quad\mbox{on}\ \partial\Omega,
$$
and $\underline{U}(x_{0})=\max_{x\in\overline{\Omega}}U(x)$,
then $h(x_{0},\underline{u}(x_{0}))\ge 0$. \vspace{1mm}
\item
If $\overline{U}\in C^{2}(\Omega )
\cap C^{1}(\overline{\Omega })$ satisfies
$$
\Delta \overline{U}+
h(x,\overline{U})\le 0
\quad\mbox{in}\ \Omega,\quad
\partial_{\nu}\overline{U}\ge 0
\quad\mbox{on}\ \partial\Omega,
$$
and $\overline{U}(x_{0})=\min_{x\in\overline{\Omega}}\overline{U}(x)$,
then $h(x_{0},\overline{U}(x_{0}))\le 0$.
\end{enumerate}
\end{lem}
In this section, we show that
the second situation (ii) of Theorem \ref{limthm}
never occurs 
in the full cross-diffusion limit.
\begin{prop}\label{CSprop}
All solutions of \eqref{CS} consist of 
three constant solutions$:$
$w=\delta a_{1}/b_{1};$
$w=0;$
$w=-\gamma a_{2}/c_{2}$.
\end{prop}

\begin{proof}
Let $w$ be any weak solution of \eqref{CS}.
Then the elliptic regularity theory (e.g., \cite{GT})
ensures that
$w\in W^{2,p}(\Omega )$ for any $p\in (1,\infty)$.
Furthermore, the Sobolev embedding theorem
leads to $w\in C^{1,\theta}(\overline{\Omega})$
for any $\theta\in (0,1)$.

Suppose for contradiction that 
there exists a nonconstant solution $w$ of \eqref{CS}.
Substituting \eqref{fgdef} into \eqref{CS}, 
one can see that $w$ satisfies
\begin{equation}\label{CS-2}
\begin{cases}
d\Delta w +\dfrac{w_{+}}{\delta}
\biggl(a_{1}-\dfrac{b_{1}w_{+}}{\delta }\biggr)
-w_{-}\biggl(
a_{2}-\dfrac{c_{2}w_{-}}{\gamma }\biggr)=0
\quad&\mbox{in}\ \Omega,\\
\partial_{\nu}w=0\quad&\mbox{on}\ \partial\Omega
\end{cases}
\end{equation}
and the integral constraint
\begin{equation}\label{CS-3}
\displaystyle\int_{\Omega}
w_{+}\biggl(a_{1}-\dfrac{b_{1}w_{+}}{\delta}\biggr)=0
\end{equation}
because $w_{+}w_{-}=0$ in $\Omega$.
Suppose that $\Omega_{+}:=\{\,x\in\Omega\,:\,w(x)>0\,\}$ is not empty.
Let $x^{*}\in\Omega_{+}$ be a maximum point of $w$, namely,
$w(x^{*})=\max_{x\in\overline{\Omega}}w(x)$.
We observe that $w$ is of class $C^{2}$ in $\Omega_{+}$ and
$w=0$ on $\Omega\cap\partial\Omega_{+}$.
Then applying (i) of Lemma \ref{MPlem} to \eqref{CS-3}, one can see
$$0\le\dfrac{w_{+}(x^{*})}{\delta}
\biggl(a_{1}-\dfrac{b_{1}w_{+}(x^{*})}{\delta }\biggr)
-w_{-}(x^{*})\biggl(
a_{2}-\dfrac{c_{2}w_{-}(x^{*})}{\gamma }\biggr)
=\dfrac{w_{+}(x^{*})}{\delta}
\biggl(a_{1}-\dfrac{b_{1}w_{+}(x^{*})}{\delta }\biggr).
$$
Then we know that
$0<w_{+}(x^{*})\le\delta a_{1}/b_{1}$,
and thereby,
$0<w_{+}(x)\le \delta a_{1}/b_{1}$ for any $x\in\Omega_{+}$.
Hence it follows that
$$
\dfrac{w_{+}(x)}{\delta}\biggl(
a_{1}-\dfrac{b_{1}w_{+}(x)}{\delta }\biggr)\ge 0
\quad\mbox{for any}\ x\in\overline{\Omega}_{+}.
$$
Since $w$ is not identically equal to
zero or $\delta a_{1}/b_{1}$ by our assumption, 
we see that
$$
\displaystyle\int_{\Omega}
w_{+}\biggl(a_{1}-\dfrac{b_{1}w_{+}}{\delta}\biggr)=
\displaystyle\int_{\Omega_{+}}
w_{+}\biggl(a_{1}-\dfrac{b_{1}w_{+}}{\delta}\biggr)>0.
$$
However, this contradicts \eqref{CS-3}.
Now we can say that
if $w$ is a nonconstant solution of \eqref{CS},
then $\Omega_{+}$ is empty.

Suppose that 
$\Omega_{-}:=\{\,x\in\Omega\,|\,w(x)<0\,\}$ is not empty.
By the boundary condition and \eqref{CS-3},
we integrate the elliptic equation of \eqref{CS-2} over $\Omega$
to obtain
\begin{equation}\label{CS-4}
\displaystyle\int_{\Omega}w_{-}\biggl(
a_{2}-\dfrac{c_{2}w_{-}}{\gamma }\biggr)=0.
\end{equation}
Let $x_{*}\in\Omega_{-}$ be a minimum point of $w$, namely,
$w(x_{*})=\min_{x\in\overline{\Omega}}w(x)$.
By a similar application of (ii) of Lemma \ref{MPlem} to
\eqref{CS-3}, we see that
$$
w_{-}(x_{*})\biggl(
a_{2}-\dfrac{c_{2}w_{-}(x_{*})}{\gamma }\biggr)\ge0,
$$
and then,
$0<w_{-}(x)\le \gamma a_{2}/c_{2}$
for any 
$x\in\Omega_{-}$.
Since $w$ does not identically equal to zero or $-\gamma a_{2}/c_{2}$ by our
assumption, then 
$$
\displaystyle\int_{\Omega}w_{-}\biggl(
a_{2}-\dfrac{c_{2}w_{-}}{\gamma }\biggr)
=
\displaystyle\int_{\Omega_{-}}w_{-}\biggl(
a_{2}-\dfrac{c_{2}w_{-}}{\gamma }\biggr)
>0.
$$
This contradicts \eqref{CS-4}.
Consequently, 
the proof by contradiction enables us to conclude that
\eqref{CS} possesses no nonconstant solution.
We complete the proof of Proposition \ref{CSprop}.
\end{proof}
Proposition \ref{CSprop} implies that
all solutions of \eqref{CS}
are corresponding to 
constant solutions of \eqref{SKT} such that $uv=0$; 
$(u,v)=(0,0)$,
$(a_{1}/b_{1}, 0)$
and
$(0, a_{2}/c_{2})$.
We recall that the situation (ii) of Theorem \ref{limthm}
asserts that $w$ is a sign-changing solution of \eqref{CS}.
Then, together with Proposition \ref{CSprop}, 
we obtain the following result:
\begin{cor}
The situation {\rm (ii)} of Theorem \ref{limthm} cannot occur.
\end{cor}

\section{Main results}
In this section, we state main results,
which are concerned with the set of nonconstant solutions
of the limiting system \eqref{IS}.
For a simple expression of \eqref{IS}, we set
\begin{equation}\label{uvdef}
u=u(w,\tau ):=\dfrac{\sqrt{w^{2}+4\gamma\delta\tau}+w}{2\delta},\quad
v=v(w,\tau ):=\dfrac{\tau }{u}=
\dfrac{\sqrt{w^{2}+4\gamma\delta\tau}-w}{2\gamma}.
\end{equation}
It is noted that 
\eqref{uvdef} gives an injection
from $(w,\tau)\in \mathbb{R}\times\mathbb{R}_{+}$
to
$(u,v)\in\mathbb{R}_{+}\times\mathbb{R}_{+}$,
where $\mathbb{R}_{+}:=(0,\infty)$, 
and the inverse of $(u(w,\tau ), v(w,\tau ))$ is given by
$$
w(u,v) =\delta u-\gamma v,
\quad
\tau (u,v)=uv.
$$

Then the limiting system \eqref{IS} is expressed as
\begin{subequations}\label{ISS}
\begin{equation}\label{ISS-1} 
d\Delta w+f(u,v)-
\gamma g(u,v)=0
\quad \mbox{in}\ \Omega,\quad \tau >0,
\end{equation}
subject to the homogeneous Neumann boundary condition
\begin{equation}\label{ISS-2}
\partial_{\nu}w=0\quad\mbox{on}\ \partial\Omega
\end{equation}
with the integral constraint
\begin{equation}\label{ISS-3}
\displaystyle\int_{\Omega }
f(u,v)=0.
\end{equation}
\end{subequations}
Obviously, any solution $(w,\tau)$ of \eqref{ISS} also satisfies
\begin{equation}\label{intg}
\int_{\Omega}g(u,v)=0,
\end{equation}
which comes from the integration of \eqref{ISS-1} over $\Omega$
using \eqref{ISS-2} and \eqref{ISS-3}.

Here we note the set of constant solutions of \eqref{ISS}.
Hereafter
$(A,B,C)$ will be denoted by
$$
A:=\dfrac{a_{1}}{a_{2}},\quad
B:=\dfrac{b_{1}}{b_{2}},\quad
C:=\dfrac{c_{1}}{c_{2}}.
$$
In the weak competition case 
$C<A<B$
or the strong competition case 
$B<A<C$,
\eqref{SKT} admits a unique positive constant solution
$(u^{*}, v^{*})$
represented as \eqref{const}.
Hence \eqref{uvdef} induces that
the limiting system \eqref{ISS} 
admits a unique constant solution
satisfying $\tau >0$;
\begin{equation}\label{wsdef}
(w^{*}, \tau^{*})=(\delta u^{*}-\gamma v^{*}, u^{*}v^{*})
\end{equation}
if $C<A<B$ or $B<A<C$.
Our interest is the set of nonconstant solutions of \eqref{ISS}.
Here we call $(w,\tau )$  a nonconstant solution
of \eqref{ISS} when $w\in C^{2}(\overline{\Omega})$ and $\tau$
satisfy \eqref{ISS} and $w(x)$ is a nonconstant function.

The first result is concerned with 
a priori estimate of all solutions
of \eqref{ISS}.
In addition,
the result asserts that
a large region of $d$ wipes out
any nonconstant solution.
\begin{thm}\label{nonexthm}
There exists a positive constant
$C^{*}=C^{*}(a_{i}, b_{i}, c_{i}, \gamma, \delta)$ such that
any solution $(w,\tau)$ of \eqref{ISS} satisfies
$$
\|w\|_{\infty}\le C^{*}\quad\mbox{and}\quad
\tau\le\min\biggl\{\,\dfrac{a_{1}^{\,2}}{4b_{1}c_{1}},
\dfrac{a_{2}^{\,2}}{4b_{2}c_{2}}\,\biggr\}\,(\,=:\overline{\tau}\,).
$$
Furthermore,
there exists $\overline{d}=\overline{d}(a_{i}, b_{i}, c_{i}, \gamma, \delta)>0$ 
such that \eqref{ISS} does not admit any nonconstant solution
if $d>\overline{d}$.
\end{thm}

The next result gives sufficient intervals of $d$ for
the existence of nonconstant solutions of \eqref{ISS}
in the weak or strong competition case with a couple of additional conditions.
To express an essential condition for the existence of
nonconstant solutions, we introduce the following function:
\begin{equation}\label{Ddef}
D(a_{i}, b_{i}, c_{i}, \gamma):=
\gamma b_{2}(A-B)(B+C-2A)+c_{2}(C-A)\{\,
A(B+C)-2BC\,\}.
\end{equation}

\begin{thm}\label{exthm}
Assume the weak competition $C<A<B$
or the strong competition $B<A<C$.
Suppose further that
\begin{equation}\label{bifcond}
D(a_{i}, b_{i}, c_{i}, \gamma)>0
\quad \mbox{and}\quad
\dfrac{u^{*}}{v^{*}}\neq\dfrac{\gamma}{\delta}.
\end{equation}
Then there exists a sequence
$\{d^{(j)}\}$ with 
$$0\leftarrow\cdots\le d^{(j+1)}\le d^{(j)}\le\cdots
\le d^{(2)}\le d^{(1)}\le\overline{d}$$
such that
\eqref{ISS} admits at least one nonconstant solution
in the following case {\rm (i)} or {\rm (ii)}:
\begin{enumerate}[{\rm (i)}]
\item
$C<A<B$ and 
$d\in (d^{(j+1)}, d^{(j)})$ 
and $j$ is odd;
\item
$B<A<C$ and
$d\in (d^{(j+1)}, d^{(j)})$ 
and $j$ is even.
\end{enumerate}
\end{thm}

\psfrag{0}[t]{{$0$}}
\psfrag{u}[t]{{{\footnotesize $u$}}}
\psfrag{t}[t]{{$\tau$}}
\psfrag{1}[t]{{$d^{(1)}$}}
\psfrag{2}[t]{{$d^{(2)}$}}
\psfrag{3}[t]{{$d^{(3)}$}}
\psfrag{c}[t]{{$\tau^{*}$}}
\psfrag{d}[t]{{$d_{2}$}}
\psfrag{p}[t]{{$\varGamma_{1}^{\pm}$}}
\psfrag{q}[t]{{$\varGamma_{2}^{\pm}$}}
\psfrag{r}[t]{{$\varGamma_{3}^{\pm}$}}
\psfrag{s}[t]{{$\tau_{0}$}}

\begin{figure}
\centering
\subfigure[$C<A<B$]{
\includegraphics*[scale=0.48]{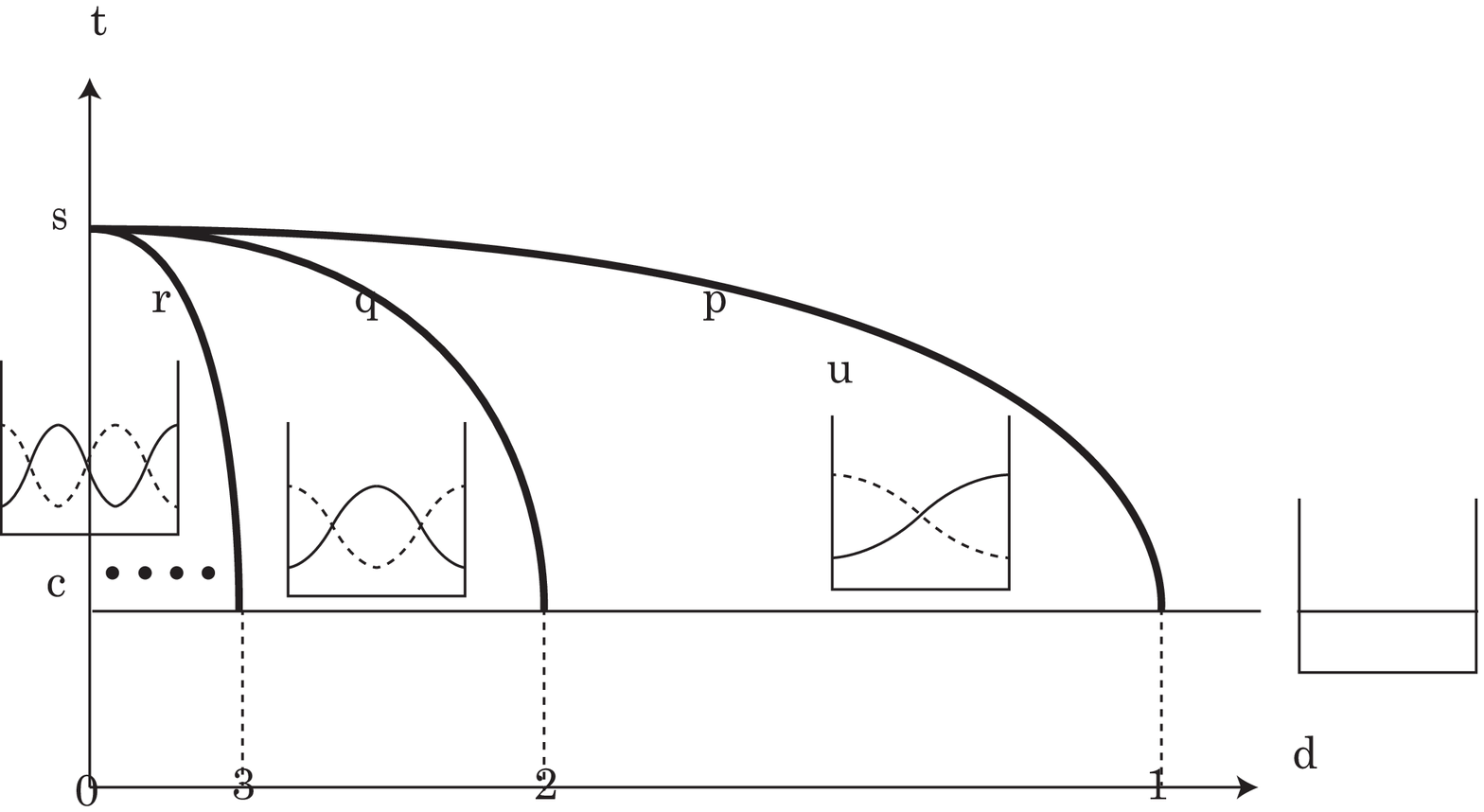}
\label{figa}}
\subfigure[$B<A<C$]{
\includegraphics*[scale=0.48]{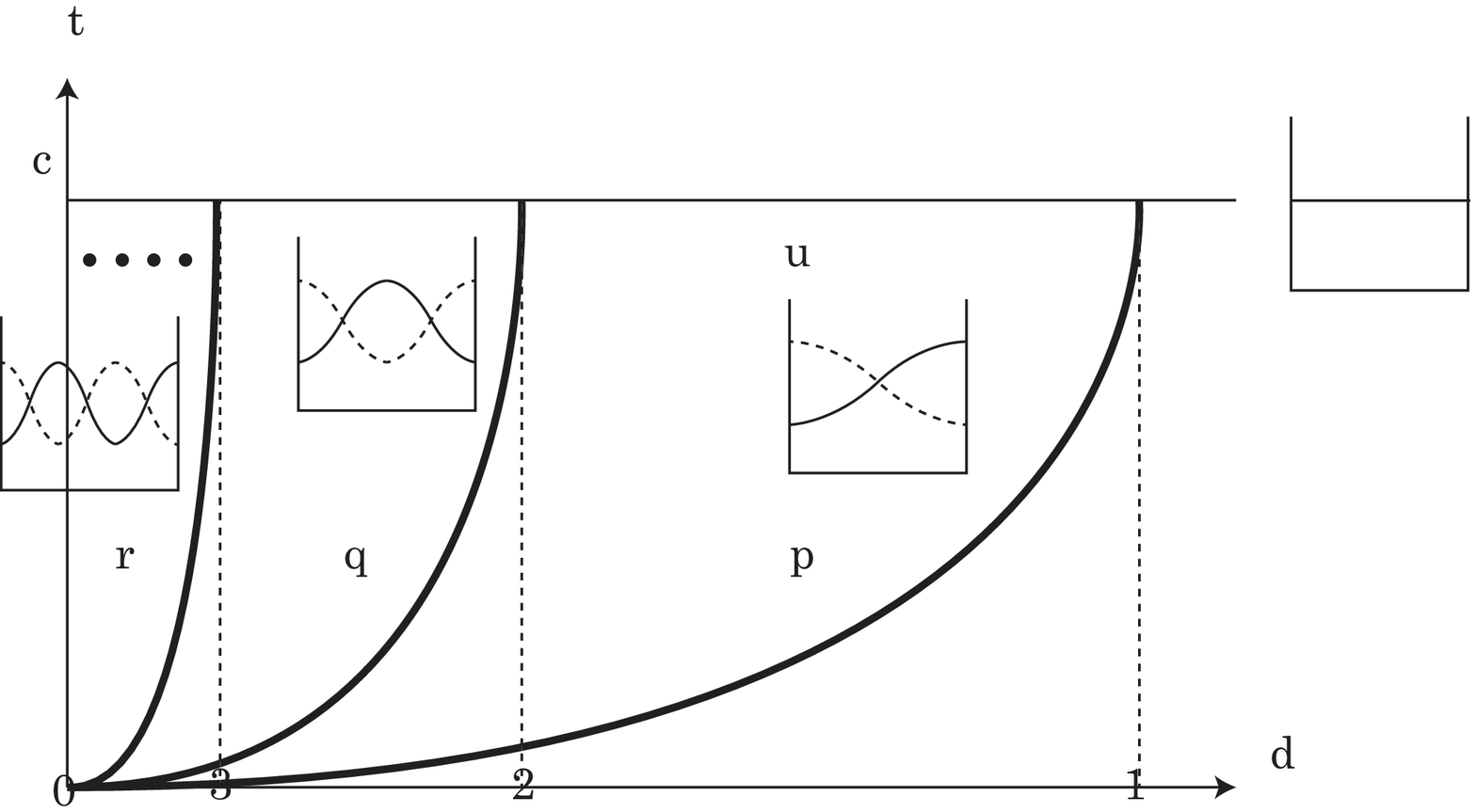}
\label{figb}
}

\caption{
Possible bifurcation diagram}
\label{fig1}
\end{figure}

In particular, for the one-dimensional case $\Omega =(0,1)$,
we show  more detailed information on the set of 
nonconstant solutions of \eqref{ISS}.
To state the global bifurcation structure of 
nonconstant solutions of \eqref{ISS} with $\Omega =(0,1)$,
we set 
$$
\mathcal{S}^{+}_{j}
:=\left\{\,(d, u, \tau)\in\mathbb{R}_{+}\times X\,:\,
\begin{array}{l}
(d,u, \tau)\
\mbox{satisfies\ \eqref{ISS}\ with\ $\Omega =(0,1)$\ and}\\
(-1)^{i-1}u'(x)>0\ \mbox{for}\ x\in
((i-1)/j, i/j)\ (i=1,2,\ldots,j)
\end{array}
\,
\right\}
$$
and
$$
\mathcal{S}^{-}_{j}
:=\left\{\,(d,u, \tau)\in\mathbb{R}_{+}\times X\,:\,
\begin{array}{l}
(d,u, \tau)\
\mbox{satisfies\ \eqref{ISS}\ with\ $\Omega =(0,1)$\ and}\\
(-1)^{i-1}u'(x)<0\ \mbox{for}\ x\in
((i-1)/j, i/j)\ (i=1,2,\ldots,j)
\end{array}
\,
\right\}
$$
for each $j\in\mathbb{N}$,
where $X:=C^{1}(\overline{\Omega})\times\mathbb{R}$.
Our aim is to construct
the global branch of nonconstant solutions,
contained in $\mathcal{S}^{\pm}_{j}$, 
that bifurcates from
the positive constant solution
$(u^{*},\tau^{*})$ at $d=d^{(j)}$ 
and reaches
a singular limit as $d\searrow 0$
in the weak or strong competition case
with $D(a_{i}, b_{i}, c_{i}, \gamma )>0$.
See also Figure 1.
\begin{thm}\label{1dimthm}
Assume the weak competition $C<A<B$ or 
the strong competition $B<A<C$.
Assume further that $D(a_{i}, b_{i}, c_{i}, \gamma )>0$.
Then for each $j\in\mathbb{N}$,
there exists a pair of connected sets
$\varGamma^{+}_{j}\,(\,\subset\mathcal{S}^{+}_{j}\,)$ 
and $\varGamma^{-}_{j}\,(\,\subset\mathcal{S}^{-}_{j}\,)$
with the following properties:
\begin{enumerate}[{\rm (i)}]
\item
$\varGamma^{+}_{j}$ bifurcates from the positive constant solution
branch $\{\,(d, u^{*},\tau^{*})\,:\,d>0\,\}$ at $d=d^{(j)}$;
\item
$\varGamma^{+}_{j}$ reaches a singular limit 
$(d, \tau)=(0, \tau_{0})$ with some $\tau_{0}\ge 0$.
Furthermore,
if $C<A<B$, then $\tau_{0}>0$,
whereas if $B<A<C$, then $\tau_{0}=0$;
\item
$\varGamma_{j}^{-}=\{\,(d, u(\,\sbt\,+1/j), \tau)\,:\,
(d, u, \tau)\in\varGamma^{+}_{j}\,\}$,
where $u(\,\sbt\,+1/j)$ is regarded as a periodic extension.
\end{enumerate}
\end{thm}
Here we should refer to a recent numerical result
by Breden, Kuehn and Soresina \cite{BKS}, which
numerically exhibits the bifurcation diagram 
of solutions of \eqref{SKT}
with $$(\alpha, \beta, a_{1}, a_{2}, b_{1}, b_{2}, c_{1}, c_{2})=
\biggl(100, 100, \dfrac{15}{2}, \dfrac{16}{7}, 4,1,6,2\biggr)$$
regarding $d=d_{1}=d_{2}$ as a bifurcation parameter.
It is easy to check that the above setting belongs to
the weak competition and satisfies $D(a_{i},b_{i}, c_{i}, 1)=17/64>0$.
Comparing the bifurcation diagram \cite[Figure 11]{BKS}
with Theorem \ref{1dimthm}, 
it can be seen that the set of solutions of the limit system \eqref{ISS}
gives a good approximation of that of \eqref{SKT} 
when both cross-diffusion coefficients $\alpha$ and $\beta$
are sufficiently large.

\section{A priori estimate}
This section is devote to the proof of Theorem \ref{nonexthm}.
We begin with 
a priori estimate for any solution
of \eqref{ISS}:
\begin{lem}\label{aprlem}
There exists a positive constant
$C^{*}=C^{*}(a_{i}, b_{i}, c_{i}, \gamma, \delta)$ such that
any solution $(w,\tau)$ of \eqref{ISS} satisfies
$\|w\|_{\infty}\le C^{*}$
and
$\tau\le\overline{\tau}$.
\end{lem}
\begin{proof}
By \eqref{fgdef},
the integral constraint \eqref{ISS-3} is equivalent to
$$
\displaystyle\int_{\Omega}
\{\,u(a_{1}-b_{1}u)-c_{1}\tau\,\}=0.
$$
By the nonnegativity of $u$, one can see that
\begin{equation}\label{tauest1}
\tau=\dfrac{1}{c_{1}|\Omega |}\displaystyle\int_{\Omega }
u(a_{1}-b_{1}u)\le\dfrac{a_{1}^{\,2}}{4b_{1}c_{1}}.
\end{equation}
From \eqref{intg}, we can deduce that
any solution $(w,\tau)$ of \eqref{IS} satisfies
\begin{equation}\label{tauest2}
\tau\le\dfrac{a_{2}^{\,2}}{4b_{2}c_{2}}
\end{equation}
by a similar way to get \eqref{tauest1}.
Hence \eqref{tauest1} and \eqref{tauest2} give
the required a priori estimate for the $\tau$ component.

Setting $v=\tau /u$ in the nonlinear term of \eqref{ISS-1},
we introduce a function $h(u,\tau )$ as 
\begin{equation}\label{hdef}
h(u,\tau):=
u(a_{1}-b_{1}u)-c_{1}\tau
-\dfrac{\gamma\tau}{u}
\biggl(a_{2}-b_{2}u-\dfrac{c_{2}\tau}{u}\biggr).
\end{equation}
It is noted that \eqref{ISS-1} is equivalent to
$$
d\Delta w+h(u,\tau )=0\quad\mbox{in}\ \Omega,\quad \tau>0.
$$
We remark that $\lim_{u\searrow 0}h(u,\tau )=\infty$
and $\lim_{u\to\infty}h(u,\tau )=-\infty$
for each $\tau >0$.
Obviously $h(u,\tau )$ has at lease one,
and at most three zeros
on $\{\,u>0\,\}$ for each $\tau >0$.
In what follows,
the least zero and the greatest zero of
$h(u,\tau )$
$(u>0)$
will be denoted by $z_{1}(\tau )$ and $\overline{z}(\tau )$, respectively.
Our first task for the a priori estimate of $\|w\|_{\infty}$
is to derive a lower bound of $z_{1}(\tau )$ and
an upper bound of $\overline{z}(\tau )$.

In the case when $\gamma b_{2}>c_{1}$,
we observe 
\begin{equation}\label{hdef2}
h(u,\tau )=u(a_{1}-b_{1}u)+(\gamma b_{2}-c_{1})\tau
+\dfrac{\gamma\tau}{u^{2}}(c_{2}\tau-a_{2}u).
\end{equation}
It is easily verified that
\begin{equation}\label{1term}
u(a_{1}-b_{1}u)+(\gamma b_{2}-c_{1})\tau
\begin{cases}
>0\quad &\mbox{for}\ u\in (0,p_{1}(\tau )\tau),\\
<0\quad &\mbox{for}\ u\in (p_{1}(\tau )\tau,\infty)
\end{cases}
\end{equation}
with 
$$
p_{1}(\tau ):=\dfrac{2(\gamma b_{2} -c_{1})}
{\sqrt{a_{1}^{\,2}+4b_{1}(\gamma b_{2}-c_{1})\tau}-a_{1}}.
$$
and
\begin{equation}\label{2term}
\dfrac{\gamma\tau}{u^{2}}
(c_{2}\tau-a_{2}u)
\begin{cases}
>0\quad &\mbox{for}\ u\in (0, c_{2}\tau/a_{2}),\\
<0\quad &\mbox{for}\ u\in (c_{2}\tau/a_{2}, \infty).
\end{cases}
\end{equation}
It follows from \eqref{hdef2}-\eqref{2term} that,
in case $\gamma b_{2}>c_{1}$,
\begin{equation}\label{z1}
\min\biggl\{\,p_{1}(\tau )\tau, \dfrac{c_{2}}{a_{2}}\tau\,\biggr\}
\le z_{1}(\tau )\le\overline{z}(\tau )\le
\max\biggl\{\,p_{1}(\tau )\tau, \dfrac{c_{2}}{a_{2}}\tau\,\biggr\}
\quad
\mbox{for any}\ \tau >0.
\end{equation}
Obviously, if $\gamma b_{2}=c_{1}$,
then
\begin{equation}\label{z2}
\min\biggl\{\,\dfrac{a_{1}}{b_{1}}, \dfrac{c_{2}}{a_{2}}\tau\,\biggr\}
\le z_{1}(\tau )\le\overline{z}(\tau )\le
\max\biggl\{\,\dfrac{a_{1}}{b_{1}}, \dfrac{c_{2}}{a_{2}}\tau\,\biggr\}
\quad
\mbox{for any}\ \tau >0.
\end{equation}

In the case when $\gamma b_{2}<c_{1}$,
we observe that
\begin{equation}\label{fterm}
u(a_{1}-b_{1}u)\begin{cases}
>0\quad&\mbox{for}\ u\in (0, a_{1}/b_{1}),\\
<0\quad&\mbox{for}\ u\in (a_{1}/b_{1}, \infty)
\end{cases}
\end{equation}
and
\begin{equation}\label{sterm}
(\gamma b_{2}-c_{1})\tau+\dfrac{\gamma\tau}{u^{2}}
(c_{2}\tau-a_{2}u)
\begin{cases}
>0\quad &\mbox{for}\ u\in (0,p_{2}(\tau )\tau ),\\
<0\quad &\mbox{for}\ u\in (p_{2}(\tau )\tau, \infty)
\end{cases}
\end{equation}
with 
$$
p_{2}(\tau ):=
\dfrac{2\gamma c_{2}}
{\sqrt{(\gamma a_{2})^{2}+4(c_{1}-\gamma b_{2})c_{2}\tau}+\gamma a_{2}}.
$$
In view of \eqref{hdef2}, we know from
\eqref{fterm} and \eqref{sterm} that,
in case $\gamma b_{2}<c_{1}$,
\begin{equation}\label{z3}
\min\biggl\{\,\dfrac{a_{1}}{b_{1}}, p_{2}(\tau )\tau\,\biggr\}
\le z_{1}(\tau )\le\overline{z}(\tau )\le
\max\biggl\{\,\dfrac{a_{1}}{b_{1}}, p_{2}(\tau )\tau\,\biggr\}
\quad
\mbox{for any}\ \tau >0.
\end{equation}

Suppose that $(w,\tau )$ is any solution of \eqref{ISS}.
Let $x_{*}\in\overline{\Omega }$ be a minimum point of 
$w$, namely,
$w(x_{*})=\min_{x\in\overline{\Omega }}w(x)$.
Then, the application of (ii) of Lemma \ref{MPlem}
to \eqref{ISS-1}-\eqref{ISS-2} implies
$$
h(u(x_{*}), \tau )\le 0,
$$
which leads to
$$
z_{1}(\tau )\le u(x_{*})=\dfrac{2\gamma \tau}
{\sqrt{w(x_{*})^{2}+4\gamma\delta\tau}-w(x_{*})}.
$$
Therefore, we see that
$w_{-}(x):=-\min\{\,w(x),0\,\}$
satisfies
$$
w_{-}(x_{*})=\dfrac{|w(x_{*})|-w(x_{*})}{2}<
\dfrac{\sqrt{w(x_{*})^{2}+4\gamma\delta\tau}-w(x_{*})}{2}
\le\dfrac{\gamma\tau}{z_{1}(\tau )}.
$$
From \eqref{z1}, \eqref{z2} and \eqref{z3}, 
we can find a positive constant 
$C_{1}=C_{1}(a_{i}, b_{i}, c_{i}, \gamma)$
such that
\begin{equation}\label{w-}
w_{-}(x_{*})\le C_{1}\quad\mbox{for any}\ 
\tau\in (0,\overline{\tau}\,].
\end{equation}

Let $x^{*}\in\overline{\Omega}$ be a maximum point of $w$;
$w(x^{*})=\max_{x\in\overline{\Omega }}w(x)$.
By applying (i) of Lemma \ref{MPlem} to \eqref{ISS-1}-\eqref{ISS-2},
we see that
$h(u(x^{*}), \tau )\ge 0$,
which yields
$$
u(x^{*})=\dfrac{\sqrt{w(x^{*})^{2}+4\gamma\delta\tau}+w(x^{*})}
{2\delta }\le \overline{z}(\tau ).
$$
Then
$w_{+}(x):=\max\{\,w(x),0\,\}$
satisfies
$$
w_{+}(x^{*})=\dfrac{|w(x^{*})|+w(x^{*})}{2}\le
\dfrac{\sqrt{w(x^{*})^{2}+4\gamma\delta\tau}+w(x^{*})}{2}
=\delta u(x^{*})\le\delta \overline{z}(\tau ).
$$
From
\eqref{z1}, \eqref{z2} and \eqref{z3},
we can find a positive constant
$C_{2}=C_{2}(a_{i}, b_{i}, c_{i}, \gamma, \delta )$
such that
\begin{equation}\label{w+}
w_{+}(x^{*})\le C_{2}
\quad\mbox{for any}\ 
\tau\in (0,\overline{\tau}\,].
\end{equation}
It follows from \eqref{w-} and \eqref{w+} that
$\|w\|_{\infty}\le C^{*}:=\max\{\,C_{1}, C_{2}\,\}$.
The proof of Lemma \ref{aprlem} is complete.
\end{proof}
With the aid of the elliptic regularity theory,
Lemma \ref{aprlem} leads to
the following a priori estimate of solutions of 
\eqref{ISS}.
\begin{cor}\label{Mcor}
For any $\varepsilon >0$,
there exists a positive constant
$C_{1}^{*}=C_{1}^{*}(\varepsilon, a_{i}, b_{i}, c_{i}, \gamma, \delta)$
which is independent of $\tau$
such that
if $d\ge\varepsilon$,
then
any solution $(w,\tau )$ of \eqref{ISS} satisfies
$\|w\|_{C^{1}(\overline{\Omega})}<C_{1}$.
\end{cor}

\begin{proof}
By the combination of  Lemma \ref{aprlem}
and the elliptic regularity theory,
we find a positive constant 
$C_{0}=C_{0}(p, a_{i}, b_{i}, c_{i}, \gamma, \delta)$ such that
any solution $(w,\tau )$ of \eqref{ISS} satisfies
$d\|w\|_{W^{2,p}}\le C_{0}$ for any $p>1$.
Hence the Sobolev embedding theorem ensures
$C_{1}^{*}$ fulfilling the required estimate.
\end{proof}

The next result asserts the nonexistence
of nonconstant solutions of \eqref{ISS}
when $d$ is sufficiently large.
\begin{lem}\label{nonexlem}
There exists a positive constant
$\overline{d}=\overline{d}
(a_{i}, b_{i}, c_{i}, \gamma, \delta)$
which is independent of $\tau$ such that
\eqref{ISS} does not have any nonconstant
solution
if $d>\overline{d}$.
\end{lem}

\begin{proof}
For any nonconstant solution $(w,\tau )$ of \eqref{ISS},
let $(u,v)$ be as in \eqref{uvdef}.
Then it follows that
$$
\begin{cases}
-d(\delta\Delta u -\gamma\Delta v)=f(u,v)-\gamma g(u,v)
\quad&\mbox{in}\ \Omega,\\
\partial_{\nu}u=\partial_{\nu}v=0
\quad&\mbox{on}\ \partial\Omega.
\end{cases}
$$
By taking the $L^{2}(\Omega )$ inner product of the elliptic equation with
$$u-\overline{u}\quad\mbox{and}\quad v-\overline{v},
\quad\mbox{where}\quad
\overline{u}:=\dfrac{1}{|\Omega|}\displaystyle\int_{\Omega }u,
\quad
\overline{v}:=\dfrac{1}{|\Omega|}\displaystyle\int_{\Omega }v,
$$
we have
$$
d\biggl(
\delta\|\nabla u\|_{2}^{2}-\gamma\displaystyle\int_{\Omega}
\nabla u\cdot\nabla v\biggr)
=
\displaystyle\int_{\Omega}f(u,v)(u-\overline{u})
-\gamma\displaystyle\int_{\Omega }g(u,v)(u-\overline{u})
$$
and
$$
d\biggl(\delta\displaystyle\int_{\Omega}
\nabla u\cdot\nabla v-
\gamma\|\nabla v\|_{2}^{2}\biggr)
=
\displaystyle\int_{\Omega}f(u,v)(v-\overline{v})
-\gamma\displaystyle\int_{\Omega }g(u,v)(v-\overline{v}),
$$
respectively.
Subtracting the second identity form the first one,
we get
\begin{equation}\label{id}
d\biggl(
\delta\|\nabla u\|^{2}_{2}
-(\gamma +\delta )
\displaystyle\int_{\Omega}
\nabla u\cdot\nabla v
+\gamma\|\nabla v\|^{2}_{2}
\biggr)
=
\displaystyle\int_{\Omega }
\{\,f(u,v)-\gamma g(u,v)\,\}\{\,(u-\overline{u})-(v-\overline{v})\,\}.
\end{equation}
Noting that $\int_{\Omega}(u-\overline{u})=\int_{\Omega}(v-\overline{v})=0$,
we substitute \eqref{fgdef} and $uv=\tau$ into 
$f(u,v)-\gamma g(u,v)$ to see
\begin{equation}\label{r1}
\begin{split}
\displaystyle\int_{\Omega}
&\{\,f(u,v)-\gamma g(u,v)\,\}(u-\overline{u})
=
\displaystyle\int_{\Omega}
\{\,u(a_{1}-b_{1}u)-\gamma v(a_{2}-c_{2}v)\,\}(u-\overline{u})\\
=&
a_{1}\displaystyle\int_{\Omega}(u-\overline{u})^{2}
-b_{1}\displaystyle\int_{\Omega}(u^{2}-\overline{u}^{2})(u-\overline{u})
-\gamma a_{2}\displaystyle\int_{\Omega }
(v-\overline{v})(u-\overline{u})
+\gamma c_{2}\displaystyle\int_{\Omega }
(v^{2}-\overline{v}^{2})(u-\overline{u})\\
=&
\displaystyle\int_{\Omega }
\{\,a_{1}-b_{1}(u+\overline{u})\,\}(u-\overline{u})^{2}
-\gamma\displaystyle\int_{\Omega}
\{\,a_{2}-c_{2}(v+\overline{v})\,\}(u-\overline{u})(v-\overline{v}).
\end{split}
\end{equation}
Similarly, one can obtain
\begin{equation}\label{r2}
\begin{split}
\displaystyle\int_{\Omega}
&\{\,f(u,v)-\gamma g(u,v)\,\}(v-\overline{v})\\
=&\displaystyle\int_{\Omega}\{\,a_{1}-b_{1}(u+\overline{u})\,\}
(u-\overline{u})(v-\overline{v})
-\gamma\displaystyle\int_{\Omega}
\{\,a_{2}-c_{2}(v+\overline{v})\,\}(v-\overline{v})^{2}.
\end{split}
\end{equation}
Here we remark that \eqref{uvdef} and Lemma \ref{aprlem}
ensure a positive constant 
$M=M(a_{i}, b_{i}, c_{i}, \gamma, \delta )$
such that
\begin{equation}\label{Mdef}
\|a_{1}-b_{1}(u+\overline{u})\|_{\infty}\le\dfrac{M}{2}
\quad\mbox{and}\quad
\gamma\|a_{2}-c_{2}(v+\overline{v})\|_{\infty}\le\dfrac{M}{2}.
\end{equation}
Substituting \eqref{r1} and \eqref{r2} into \eqref{id} and
using \eqref{Mdef} and the Schwarz inequality, we obtain
$$
d\biggl(
\delta\|\nabla u\|^{2}_{2}
-(\gamma +\delta )
\displaystyle\int_{\Omega}
\nabla u\cdot\nabla v
+\gamma\|\nabla v\|^{2}_{2}
\biggr)
\le M(\,\|u-\overline{u}\|^{2}_{2}+\|v-\overline{v}\|^{2}_{2}\,).
$$
Here we recall the Poincar\'e-Wirtinger inequality;
$$
\lambda_{1}\|U-\overline{U}\|^{2}_{2}\le\|\nabla U\|^{2}_{2}
\quad\mbox{for any}\ U\in H^{1}(\Omega ),
$$
where $\lambda_{1}$ represents the least positive eigenvalue
of \eqref{Lap}.
Therefore, we obtain
\begin{equation}\label{PW}
d\biggl(
\delta\|\nabla u\|^{2}_{2}
-(\gamma +\delta )
\displaystyle\int_{\Omega}
\nabla u\cdot\nabla v
+\gamma\|\nabla v\|^{2}_{2}
\biggr)
\le \dfrac{M}{\lambda_{1}}(\,\|\nabla u\|_{2}^{2}
+\|\nabla v\|^{2}_{2}\,).
\end{equation}
It follows from $v=\tau /u$ that
$$
\displaystyle\int_{\Omega}
\nabla u\cdot\nabla v=
\displaystyle\int_{\Omega}
\nabla u\cdot\nabla\biggl(\dfrac{\tau}{u}\biggr)=
-\tau\displaystyle\int_{\Omega}\biggl|
\dfrac{\nabla u}{u}\biggr|^{2}<0.
$$
Then we know from \eqref{PW} that
$$
\biggl(d\delta-\dfrac{M}{\lambda_{1}}\biggr)
\|\nabla u\|^{2}_{2}+
\biggl(d\gamma-\dfrac{M}{\lambda_{1}}\biggr)
\|\nabla v\|^{2}_{2}\le 0,
$$
which concludes
$$
d\le \overline{d}:=\max\biggl\{\,
\dfrac{M}{\gamma \lambda_{1}},
\dfrac{M}{\delta \lambda_{1}}\,\biggr\}
$$
because neither $u$ nor $v=\tau/u$ is constant.
Then we complete the proof of Lemma \ref{nonexlem}.
\end{proof}

\begin{proof}[Proof of Theorem \ref{nonexthm}]
Theorem \ref{nonexthm}  follows from 
Lemmas \ref{aprlem} and \ref{nonexlem}.
\end{proof}

\section{Existence of nonconstant solutions
for the full cross-diffusion limit}
In order to find nonconstant solutions of \eqref{ISS}
by the Leray-Schauder degree theory,
we set up a functional space $X$ and a compact operator
$F\,:\,\mathbb{R}\times X\to X$ as follows:
$
X:=C^{1}(\overline{\Omega})\times \mathbb{R}
$
and 
\begin{equation}\label{Fdef}
F(d,w,\tau):=
\left[
\begin{array}{l}
F^{(1)}(d,w,\tau)\\
F^{(2)}(d,w,\tau)
\end{array}
\right],
\end{equation}
where
\begin{equation}
\begin{split}
&F^{(1)}(d,w,\tau):=
(I-\Delta)^{-1}\left[
w+\dfrac{f(u(w,\tau ), v(w,\tau ))
-\gamma g(u(w,\tau ), v(w, \tau ))}{d}\right], \\
&F^{(2)}(d,w,\tau):=
\dfrac{1}{c_{1}|\Omega |}\int_{\Omega }
u(w, \tau )\{\,a_{1}-b_{1}u(w, \tau )\,\}.
\end{split}
\nonumber
\end{equation}
Here 
$(u(w,\tau ), v(w,\tau ))$ is as in \eqref{uvdef}
and
$(I-\Delta )^{-1}$ 
is regarded as
a composition of 
the
inverse operator of
$I-\Delta\,:\,\,W^{2,p}_{\nu}(\Omega)
\,(\,:=\{\,w\in W^{2,p}(\Omega )\,:\,
\partial_{\nu}w=0\ \mbox{on}\ \partial\Omega\,\}\,)\to L^{p}(\Omega )$
with the domain restricted to $C^{1}(\overline{\Omega})$ and
the compact embedding from 
$W^{2,p}(\Omega )$ 
into $C^{1}(\overline{\Omega })$
with $p>N$.
Then,
for any $d>0$,
each weak solution of \eqref{ISS}
is corresponding to each fixed point of $F(d,\,\sbt\,,\,\sbt\,)$.
To find fixed points with $\tau >0$,
we introduce a bounded set $S_{\eta, M}$ in $X$ as 
$$
S_{\eta, M}:=
\{\,(w,\tau)\in X\,:\,\eta<\|w\|_{C^{1}(\overline{\Omega})}<M,\
\eta<\tau<M\,\}
$$
for $0<\eta<M$.
In the following lemma,
$C_{1}^{*}$ and
$\overline{\tau}$ are positive constants
obtained in Theorem \ref{nonexthm}.
\begin{lem}\label{bddlem}
Assume $A\neq B$ and $A\neq C$.
Furthermore, 
especially in the weak competition case 
$C<A<B$ or
the strong competition case 
$B<A<C$,
assume that $(u^{*}, v^{*})$ in \eqref{const} satisfies
$u^{*}/v^{*}\neq \gamma /\delta$. 
Then, for any small $\varepsilon >0$, there exists a small 
$\eta=\eta(\varepsilon )>0$ such that
if $d\ge\varepsilon$, then
any solution $(w, \tau )$ of \eqref{ISS} satisfies
$(w,\tau )\not\in \partial S_{\eta,M}$,
where
$M$ is any constant satisfying $M>C_{1}^{*}$ and $M>\overline{\tau}$.
\end{lem}

\begin{proof}
It follows from Theorem \ref{nonexthm} and Corollary \ref{Mcor} that 
any solution $(w,\tau )$ of \eqref{ISS} with 
$d\ge\varepsilon$ satisfies
$(w, \tau )\in B_{M}:=
\{\,(w,\tau )\in X\,:\,
\|w\|_{C^{1}(\overline{\Omega})}<M,\quad
0\le \tau <M\,\}$
if $M>C_{1}^{*}$ and $M>\overline{\tau}$.
Obviously,
semitrivial solutions $(w,\tau )=(\delta a_{1}/b_{1}, 0)$,
$(-\gamma a_{2}/c_{2}, 0)$
and the trivial solution $(w,\tau )=(0,0)$
are not contained in the closure of
$S_{\eta, M}$ because $\eta >0$.

Then our task is to prove that for any small $\varepsilon>0$,
there exists $\eta =\eta (\varepsilon )>0$ such that
any solution $(w,\tau)$ of \eqref{ISS} with $d\ge\varepsilon$
satisfies
$\|w\|_{C^{1}(\overline{\Omega})}>\eta$
and $\tau >\eta$.
Suppose for contradiction that there exists
$\hat{\varepsilon}>0$
such that 
for any small $\eta >0$,
there exists some $\hat{d}=\hat{d}(\eta )\ge\hat{\varepsilon}$ such that
\eqref{ISS} with $d=\hat{d}$ has a solution
$(\hat{w}(\eta ), \hat{\tau}(\eta ))$ satisfying
$\|\hat{w}(\eta )\|_{C^{1}(\overline{\Omega })}\le\eta$ or $\tau\le\eta$.
By virtue of Lemma \ref{nonexlem},
we can choose a subsequence 
$(\hat{d}_{n}, w_{n}, \tau_{n})\in [\,\hat{\varepsilon}, \overline{d}\,]\times B_{M}$
of 
$\{\,(\hat{d}(\eta ), \hat{w}(\eta ), \hat{\tau}(\eta ))\,\}_{\eta >0}$
such that
$$\|w_{n}\|_{C^{1}(\overline{\Omega })}\to 0\quad\mbox{or}\quad
\tau_{n}\to 0$$
and
$\hat{d}_{n}\to d_{\infty}$ with some $d_{\infty}\in
[\,\hat{\varepsilon}, \overline{d}\,]$
as $n\to\infty$.

Suppose that
$\|w_{n}\|_{C^{1}(\overline{\Omega })}\to 0$
and $\limsup_{n\to\infty}\tau_{n}>0$. We
may assume $\tau_{n}\to\tau_{0}>0$ by 
passing to a subsequence if necessary.
By \eqref{uvdef}, one can see that
\begin{equation}\label{unvn}
\begin{split}
u_{n}:=
\dfrac{\sqrt{w_{n}^{2}+4\gamma\delta\tau_{n}}+w_{n}}{2\delta}
\to\sqrt{\dfrac{\gamma\tau_{0}}{\delta}}=:u_{\infty}
\quad&\mbox{in}\ C^{1}(\overline{\Omega }),\\
v_{n}:=
\dfrac{\sqrt{w_{n}^{2}+4\gamma\delta\tau_{n}}-w_{n}}{2\gamma}
\to\sqrt{\dfrac{\delta\tau_{0}}{\gamma}}=:v_{\infty}
\quad&\mbox{in}\ C^{1}(\overline{\Omega}).
\end{split}
\nonumber
\end{equation}
Setting $n\to\infty$ in \eqref{ISS-3} and \eqref{intg}, we get
$f(u_{\infty},v_{\infty})=g(u_{\infty}, v_{\infty})=0$.
Hence it follows that 
$(u_{\infty}, v_{\infty})=(u^{*}, v^{*})$.
However, this is impossible under the assumption
$u^{*}/v^{*}\neq \gamma /\delta$.

Suppose that
$\|w_{n}\|_{C^{1}(\overline{\Omega })}\to 0$
and $\tau_{n}\to 0$.
We set
\begin{equation}\label{below}
\widetilde{w}_{n}(x):=
\dfrac{w_{n}(x)}{\|w_{n}\|_{\infty}}.
\end{equation}
Substituting \eqref{fgdef} and \eqref{uvdef} into \eqref{ISS}-\eqref{intg}
and dividing the resulting expressions by $\|w_{n}\|_{\infty}$,
one can see that
\begin{equation}\label{tildeell}
\begin{cases}
\hat{d}_{n}\Delta\widetilde{w}_{n}+
\dfrac{\sqrt{w_{n}^{2}+4\gamma\delta\tau_{n}}+w_{n}}
{2\delta\|w_{n}\|_{\infty}}
\biggl(a_{1}-b_{1}
\dfrac{\sqrt{w_{n}^{2}+4\gamma\delta\tau_{n}}+w_{n}}
{2\delta}
-c_{1}
\dfrac{\sqrt{w_{n}^{2}+4\gamma\delta\tau_{n}}-w_{n}}
{2\gamma}\biggr) \vspace{1mm} \\
-\dfrac{\sqrt{w_{n}^{2}+4\gamma\delta\tau_{n}}-w_{n}}
{2\|w_{n}\|_{\infty}}
\biggl(a_{2}-b_{2}
\dfrac{\sqrt{w_{n}^{2}+4\gamma\delta\tau_{n}}+w_{n}}
{2\delta}
-c_{2}
\dfrac{\sqrt{w_{n}^{2}+4\gamma\delta\tau_{n}}-w_{n}}
{2\gamma}\biggr)=0\quad\mbox{in}\ \Omega,\vspace{1mm} \\
\partial_{\nu}\widetilde{w}_{n}=0
\quad\mbox{on}\ \partial\Omega
\end{cases}
\end{equation}
and 
\begin{equation}\label{tildeint}
\begin{cases}
\displaystyle\int_{\Omega}
\dfrac{\sqrt{w_{n}^{2}+4\gamma\delta\tau_{n}}+w_{n}}
{2\|w_{n}\|_{\infty}}
\biggl(a_{1}-b_{1}
\dfrac{\sqrt{w_{n}^{2}+4\gamma\delta\tau_{n}}+w_{n}}
{2\delta}
-c_{1}
\dfrac{\sqrt{w_{n}^{2}+4\gamma\delta\tau_{n}}-w_{n}}
{2\gamma}\biggr)=0,\vspace{1mm} \\
\displaystyle\int_{\Omega}
\dfrac{\sqrt{w_{n}^{2}+4\gamma\delta\tau_{n}}-w_{n}}
{2\|w_{n}\|_{\infty}}
\biggl(a_{2}-b_{2}
\dfrac{\sqrt{w_{n}^{2}+4\gamma\delta\tau_{n}}+w_{n}}
{2\delta}
-c_{2}
\dfrac{\sqrt{w_{n}^{2}+4\gamma\delta\tau_{n}}-w_{n}}
{2\gamma}\biggr)=0.
\end{cases}
\end{equation}
By applying the elliptic regularity theory to \eqref{tildeell},
we find a function $\widetilde{w}\in W^{2,p}(\Omega )$ for any $p>1$
such that
$$
\lim_{n\to\infty}\widetilde{w}_{n}=
\widetilde{w}\quad\mbox{weakly in $W^{2,p}(\Omega)$ and
strongly in $C^{1}(\overline{\Omega})$}
$$
by passing to a subsequence.
Hence it follows that $\|\widetilde{w}\|_{\infty}=1$.
Then we set $n\to\infty$ in \eqref{tildeint} 
to know
$$
\int_{\Omega}
\widetilde{w}_{+}=
\int_{\Omega}
\widetilde{w}_{-}=0,
$$
which leads to $\widetilde{w}\equiv 0$.
However, this contradicts
$\|\widetilde{w}\|_{\infty}=1$.

Suppose that $\limsup_{n\to\infty}\|w_{n}\|_{C^{1}(\overline{\Omega })}>0$
and $\tau_{n}\to 0$. 
Similarly, a usual compactness argument 
applying the elliptic regularity theory
to \eqref{ISS} ensures a function $w_{\infty}\in W^{2,p}(\Omega )$
for any $p>1$ such that
$\lim_{n\to\infty}w_{n}=w_{\infty}$ in $C^{1}(\overline{\Omega })$,
$\|w_{\infty}\|_{C^{1}(\overline{\Omega })}>0$,
and $w_{\infty}$ is a weak solution of \eqref{CS}
with $d=d_{\infty}$.
By virtue of Proposition \ref{CSprop},
we see that $w_{\infty}=\delta a_{1}/b_{1}$ or
$w_{\infty}=-\gamma a_{2}/c_{2}$ in $\Omega$.
Suppose that $w_{\infty}=\delta a_{1}/b_{1}$ in $\Omega$.
It follows that 
\begin{equation}\label{unif}
\lim_{n\to\infty}(u_{n}, v_{n})=
\biggl(\dfrac{a_{1}}{b_{1}}, 0\biggr)
\quad\mbox{uniformly in}\ \overline{\Omega},
\end{equation}
where $(u_{n}, v_{n})$ is defined by \eqref{unvn}.
From \eqref{intg}, we observe that
\begin{equation}\label{intg2}
\int_{\Omega}v_{n}(a_{2}-b_{2}u_{n}-c_{2}v_{n})=0
\quad\mbox{for any}\ n\in\mathbb{N}.
\end{equation}
Owing to the assumption $A\neq B$, we know from \eqref{unif}
that
$$
a_{2}-b_{2}u_{n}-c_{2}v_{n}>0
\quad\mbox{or}\quad
a_{2}-b_{2}u_{n}-c_{2}v_{n}<0
\quad\mbox{in}\ \Omega
$$
if $n$ is sufficiently large.
This obviously contradicts \eqref{intg2} because
$v_{n}>0$ in $\Omega$ for any $n\in\mathbb{N}$.
By a similar way, we can see that
$w_{\infty}=-\gamma a_{2}/c_{2}$ is also impossible
by the assumption $A\neq C$.
Consequently, we obtain the required assertion in Lemma \ref{bddlem}
by the contradiction argument.
\end{proof}
Therefore,
under assumptions in Lemma \ref{bddlem},
the compact nonlinear map
$F(d,\,\sbt\,,\,\sbt\,)\,:\,X\to X$
has no fixed point on $\partial S_{\eta, M}$
for any $d\ge\varepsilon$.
Hence 
the homotopy invariance of the Leray-Schauder degree
implies that
\begin{equation}\label{hom}
\mbox{deg}\,(I-F(d,\,\sbt\,,\,\sbt\,),S_{\eta,M},0)
\ \mbox{is constant for any}\ 
d\ge\varepsilon.
\end{equation}
Here we recall Lemma \ref{nonexlem} to note that all
fixed points of
$F(d,\,\sbt\,,\,\sbt\,)$
contained in $S_{\eta,M}$ are restricted
to constant solutions of \eqref{ISS} if $d>\overline{d}$.
Then, in the weak or strong competition case 
when such a constant solution is uniquely determined
by $(w^{*}, \tau^{*})$ as in \eqref{wsdef}.
The following lemma asserts that the index of
$I-F(d,\,\sbt\,,\,\sbt\,)$ at $(w^{*},\tau^{*})$
changes infinitely many times as $d\searrow 0$ 
provided $D(a_{i}, b_{i}, c_{i}, \gamma )>0$
(see \eqref{Ddef} for the definition of $D(a_{i}, b_{i}, c_{i}, \gamma )>0$).
\begin{lem}\label{indexlem}
Assume the weak competition $C<A<B$
or the strong competition $B<A<C$.
Suppose further that
$D(a_{i}, b_{i}, c_{i}, \gamma)>0$.
Then
there exists a sequence
$\{d^{(j)}\}^{\infty}_{j=1}$ with 
$$0\leftarrow\cdots\le d^{(j+1)}\le d^{(j)}\le\cdots
\le d^{(2)}\le d^{(1)}\le\overline{d}$$
such that if $C<A<B$, then
$$
{\rm ind}\,(I-F(d,\,\sbt\,,\,\sbt\,), (w^{*}, \tau^{*}))=
\begin{cases}
1\quad &\mbox{for}\ d\in (d^{(j+1)}, d^{(j)})\ \mbox{and $j$ is even},\\
-1\quad &\mbox{for}\ d\in (d^{(j+1)}, d^{(j)})\ \mbox{and $j$ is odd},
\end{cases}
$$
whereas, if $B<A<C$,
then 
$$
{\rm ind}\,(I-F(d,\,\sbt\,,\,\sbt\,), (w^{*}, \tau^{*}))=
\begin{cases}
1\quad &\mbox{for}\ d\in (d^{(j+1)}, d^{(j)})\ \mbox{and $j$ is odd},\\
-1\quad &\mbox{for}\ d\in (d^{(j+1)}, d^{(j)})\ \mbox{and $j$ is even}.
\end{cases}
$$
\end{lem}

\begin{proof}
For the sake of the calculation of 
$\mbox{ind}\,(I-F(d,\,\sbt\,,\,\sbt\,), (w^{*}, \tau^{*}))$,
the following eigenvalue problem will be considered:
\begin{equation}\label{eg}
(I-L(d))
\biggl[
\begin{array}{c}
\phi\\
\xi
\end{array}
\biggr]
=
\mu
\biggl[
\begin{array}{c}
\phi\\
\xi
\end{array}
\biggr],
\end{equation}
where $L(d)$ is a linear compact operator from $X$ to $X$ 
defined by 
the linearized operator of  
$F(d,\,\sbt\,,\,\sbt\,)$ around $(w^{*},\tau^{*})$
as follows:
$$
L(d):=F_{(w,\tau)}(d,w^{*},\tau^{*}).
$$
It follows from the index formula (see e.g., \cite[Theorem 2.8.1]{Ni})
that
\begin{equation}\label{ind}
{\rm ind}\,(I-F(d,\,\sbt\,,\,\sbt\,), (w^{*}, \tau^{*}))
=(-1)^{\sigma (d)},
\end{equation}
where $\sigma (d)$ is the number of negative eigenvalues
(counting algebraic multiplicity) of \eqref{eg}.
Hereafter, each entry of $L(d)$ will be denoted by
$$
L(d)=
\biggl[
\begin{array}{cc}
L_{11}(d) & L_{12}(d) \\
L_{21}(d) & L_{22}(d)
\end{array}
\biggr]:=\biggl[
\begin{array}{cc}
F^{(1)}_{w}(d,w^{*},\tau^{*}) & F^{(1)}_{\tau}(d, w^{*}, \tau^{*})\\
F^{(2)}_{w}(d,w^{*},\tau^{*}) & F^{(2)}_{\tau}(d, w^{*}, \tau^{*})
\end{array}
\biggr].
$$
It follows from \eqref{Fdef} that
\begin{equation}\label{Lij}
\begin{split}
&L_{11}(d)=\biggl(
1+\dfrac{
f_{u}^{*}u_{w}^{*}+f_{v}^{*}v_{w}^{*}-
\gamma (g_{u}^{*}u_{w}^{*}+g_{v}^{*}v_{w}^{*})}{d}
\biggr)\,
(I-\Delta )^{-1},\\
&L_{12}(d)=
\dfrac{
f_{u}^{*}u_{\tau }^{*}+f_{v}^{*}v_{\tau }^{*}-
\gamma (g_{u}^{*}u_{\tau }^{*}+g_{v}^{*}v_{\tau }^{*})}{d}
\,
(I-\Delta )^{-1},\\
&L_{21}(d)=
\dfrac{(a_{1}-2b_{1}u^{*})u_{w}^{*}}{c_{1}|\Omega|}
\displaystyle\int_{\Omega}\,\sbt\ \ , \quad
L_{22}(d)=
\dfrac{(a_{1}-2b_{1}u^{*})u_{\tau}^{*}}{c_{1}},
\end{split}
\end{equation}
where $f^{*}_{u}:=f_{u}(u^{*},v^{*})$,
$u_{w}^{*}:=u_{w}(w^{*},\tau^{*})$
and other notations are defined by the same manner.
Observing that
$f(u^{*},v^{*})=g(u^{*},v^{*})=0$ and
$\sqrt{w^{2}+4\gamma\delta\tau}=\delta u+\gamma v$,
one can verify 
$$
\biggl[
\begin{array}{cc}
f_{u}^{*} & f_{v}^{*}\\
g_{u}^{*} & g_{v}^{*}
\end{array}
\biggr]
=
-
\biggl[
\begin{array}{cc}
b_{1}u^{*} & c_{1}u^{*}\\
b_{2}v^{*} & c_{2}v^{*}
\end{array}
\biggr],\quad
\biggl[
\begin{array}{cc}
u^{*}_{w} & u^{*}_{\tau}\\
v^{*}_{w} & v^{*}_{\tau}
\end{array}
\biggr]
=
\dfrac{1}{\delta u^{*}+\gamma v^{*}}
\biggl[
\begin{array}{cc}
u^{*} & \gamma \\
-v^{*} & \delta
\end{array}
\biggr]
$$
by a straightforward calculation.
Substituting these expressions into \eqref{Lij}, we get
\begin{equation}\label{Lij2}
\begin{split}
&L_{11}(d)=
\biggl(
1+\dfrac{(\gamma b_{2}+c_{1})\tau^{*}-b_{1}(u^{*})^{2}-\gamma c_{2}(v^{*})^{2}}
{(\delta u^{*}+\gamma v^{*})d}\biggr)\,(I-\Delta )^{-1},\\
&L_{12}(d)=-
\dfrac{(\gamma b_{1}+\delta c_{1})u^{*}-\gamma (\gamma b_{2}+\delta c_{2})v^{*}}{(\delta u^{*}+\gamma v^{*})d}\,(I-\Delta )^{-1},\\
&
L_{21}(d)=
\dfrac{c_{1}\tau^{*}-b_{1}(u^{*})^{2}}
{c_{1}(\delta u^{*}+\gamma v^{*})|\Omega |}\displaystyle
\int_{\Omega}\,\sbt\ \ ,\quad
L_{22}(d)=
\dfrac{\gamma (c_{1}v^{*}-b_{1}u^{*})}{c_{1}(\delta u^{*}+\gamma v^{*})},
\end{split}
\end{equation}
where $a_{1}-2b_{1}u^{*}=c_{1}v^{*}-b_{1}u^{*}$ is used for
expressions of $L_{21}(d)$ and $L_{22}(d)$.
Therefore, we substitute \eqref{Lij2} into \eqref{eg} to see that
the eigenvalue problem \eqref{eg} is equivalent to
\begin{equation}\label{eg2}
\begin{cases}
-(1-\mu)\Delta\phi
-
\dfrac{(\gamma b_{2}+c_{1})\tau^{*}-b_{1}(u^{*})^{2}-\gamma c_{2}(v^{*})^{2}}
{(\delta u^{*}+\gamma v^{*})d}\phi 
\vspace{0.5mm} \\ \hspace{21.2mm}
+
\dfrac{(\gamma b_{1}+\delta c_{1})u^{*}
-\gamma (\gamma b_{2}+\delta c_{2})v^{*}}
{(\delta u^{*}+\gamma v^{*})d}\xi=\mu\phi
\quad &\mbox{in}\ \Omega,\vspace{1mm} \\
\dfrac{b_{1}(u^{*})^{2}-c_{1}\tau^{*}}
{c_{1}(\delta u^{*}+\gamma v^{*})}\overline{\phi}
+\dfrac{(\gamma b_{1}+\delta c_{1})u^{*}}
{c_{1}(\delta u^{*}+\gamma v^{*})}\xi=
\mu\xi,
\vspace{1mm} \\
\partial_{\nu}\phi=0
\quad &\mbox{on}\ \partial\Omega,
\end{cases}
\end{equation}
where $\overline{\phi}:=|\Omega |^{-1}\int_{\Omega}\phi$.
To seek for nontrivial solutions of \eqref{eg2},
we introduce the 
Fourier expansion of $\phi$ as
\begin{equation}\label{Fou}
\phi (x) =\sum^{\infty}_{j=0}
q_{j}\varPhi_{j}(x),
\end{equation}
where $\{\varPhi_{j}\}^{\infty}_{j=0}$ 
is a complete orthonormal basis in $L^{2}(\Omega )$ 
defined by \eqref{Lap}.
Substituting \eqref{Fou} into the first equation of \eqref{eg2}, we obtain
\begin{equation}\label{eg3}
\begin{split}
&\sum\limits^{\infty}_{j=0}
\biggl\{\,
(1-\mu)\lambda_{j}-\mu-
\dfrac{ (\gamma b_{2}+c_{1}) \tau^{*}-b_{1}(u^{*})^{2}
-\gamma c_{2}(v^{*})^{2} }
{ (\delta u^{*}+\gamma v^{*})d } \,\biggr\}
q_{j}\varPhi_{j} \vspace{1mm} \\
&+
\dfrac{ (\gamma b_{1}+\delta c_{1})u^{*}-\gamma 
(\gamma b_{2}+\delta c_{2}) v^{*} }
{ (\delta u^{*}+\gamma v^{*})d }\xi=0\quad\mbox{in}\ \Omega.
\end{split}
\end{equation}
Integrating \eqref{eg3} over $\Omega$, we see
$$
-\dfrac{(\gamma b_{2}+c_{1})\tau^{*}-b_{1}(u^{*})^{2}-\gamma c_{2}(v^{*})^{2}}
{(\delta u^{*}+\gamma v^{*})d}q_{0}|\Omega|^{-1/2}+
\dfrac{(\gamma b_{1}+\delta c_{1})u^{*}-\gamma (\gamma b_{2}+\delta c_{2})v^{*}}{(\delta u^{*}+\gamma v^{*})d}\xi =\mu q_{0}|\Omega|^{-1/2}.
$$
Here it is noted that $\varPhi_{0}=|\Omega|^{-1/2}$ and 
$\overline{\phi}=q_{0}|\Omega |^{-1/2}$.
Together with the second equation of \eqref{eg2},
we obtain the following equation 
which $\xi$ and the constant component of $\phi$ satisfy
\begin{equation}\label{consteg}
M(d; a_{i}, b_{i}, c_{i}, \gamma, \delta)
\left[
\begin{array}{c}
\overline{\phi}\\
\xi
\end{array}
\right]
=
\mu
\left[
\begin{array}{c}
\overline{\phi}\\
\xi
\end{array}
\right],
\end{equation}
where
$$
M(d; a_{i}, b_{i}, c_{i}, \gamma, \delta)=
\dfrac{1}{\delta u^{*}+\gamma v^{*}}
\left[
\begin{array}{ll}
-\frac{(\gamma b_{2}+c_{1})\tau^{*}-b_{1}(u^{*})^{2}-\gamma c_{2}(v^{*})^{2}}
{d}
&
\frac{(\gamma b_{1}+\delta c_{1})u^{*}-\gamma (\gamma b_{2}+\delta c_{2})v^{*}}{d} \vspace{1mm} \\
\frac{b_{1}(u^{*})^{2}-c_{1}\tau^{*}}
{c_{1}}
&
\frac{(\gamma b_{1}+\delta c_{1})u^{*}}
{c_{1}}
\end{array}
\right].
$$
In what follows, we denote by
$\mu_{0}^{-}(d)$ and $\mu_{0}^{+}(d)$
eigenvalues of \eqref{consteg} satisfying
$$\mbox{Re}\,\mu_{0}^{-}(d)\le \mbox{Re}\,\mu_{0}^{+}(d)
\quad\mbox{and}\quad
\mbox{Im}\,\mu_{0}^{-}(d)\le \mbox{Im}\,\mu_{0}^{+}(d).$$
By a straightforward computation.
one can verify that 
$$
\mu_{0}^{-}(d)\mu_{0}^{+}(d)
=|\,M(d; a_{i}, b_{i}, c_{i}, \gamma, \delta)\,|
=\dfrac{\gamma u^{*}v^{*}}{c_{1}(\delta u^{*}+\gamma v^{*})^{2}d}
(b_{1}c_{2}-b_{2}c_{1}).
$$
Hence it follows that
\begin{equation}\label{mu0}
\mu_{0}^{-}(d)\mu_{0}^{+}(d)
\begin{cases}
>0\quad &\mbox{if}\ C<A<B,\\
<0\quad &\mbox{if}\ B<A<C.
\end{cases}
\end{equation}
Here we set
$$
\sigma_{0}(d):=
\mbox{the number of negatives of}\ 
\{\,\mu_{0}^{-}(d), \mu_{0}^{+}(d)\,\}.
$$
By \eqref{mu0},
one can see that, for any $d>0$,
\begin{equation}\label{sigma0}
\sigma_{0}(d)=
\begin{cases}
0\ \ \mbox{or}\ \ 2\quad &\mbox{if}\ C<A<B,\\
1\quad &\mbox{if}\ B<A<C.
\end{cases}
\end{equation}

Taking the $L^{2}(\Omega )$ inner product of 
\eqref{eg3} with $\varPhi_{j}$,
we know that
$$
\biggl\{\,
(1 -\mu)\lambda_{j}-\mu-
\dfrac{ (\gamma b_{2}+c_{1}) \tau^{*}-b_{1}(u^{*})^{2}
-\gamma c_{2}(v^{*})^{2} }
{ (\delta u^{*}+\gamma v^{*})d } \,\biggr\}
q_{j}=0
\quad\mbox{for any}\ j\in\mathbb{N}.
$$
Then the component of $\varPhi_{j}$ of \eqref{eg3} is nontrivial
as $q_{j}\neq 0$ if
\begin{equation}\label{muj}
\mu=\mu_{j}(d):=
\dfrac{1}{\lambda_{j}+1}
\biggl(\lambda_{j}-
\dfrac{(\gamma b_{2}+c_{1})\tau^{*}-b_{1}(u^{*})^{2}-\gamma c_{2}(v^{*})^{2}}
{(\delta u^{*}+\gamma v^{*})d}\biggr).
\end{equation}
Consequently, we deduce that
all eigenvalues of $I-L(d)$ consist of
$$
\{\,\mu_{0}^{-}(d), \mu_{0}^{+}(d), \mu_{1}(d), \mu_{2}(d),
\ldots,\mu_{j}(d),\ldots\,\}.
$$
It follows from \eqref{muj} that,
if $(\gamma b_{2}+c_{1})\tau^{*}-b_{1}(u^{*})^{2}-\gamma c_{2}(v^{*})^{2}
\le 0$,
then $\mu_{j} (d)>0$ for any $d>0$.
On the other hand,
if $(\gamma b_{2}+c_{1})\tau^{*}-b_{1}(u^{*})^{2}-\gamma c_{2}(v^{*})^{2}>0$,
then,
for each $j\in\mathbb{N}$,
$\mu_{j}(d)$ is monotone increasing with respect to 
$d>0$ and satisfies
\begin{equation}\label{mjcut}
\mu_{j}(d)\begin{cases}
<0\quad &\mbox{for}\ d\in (0, d^{(j)}),\\
=0\quad &\mbox{for}\ d=d^{(j)},\\
>0\quad &\mbox{for}\ d\in (d^{(j)},\infty),
\end{cases}
\end{equation}
where
\begin{equation}\label{bif0}
d^{(j)}:=
\dfrac{(\gamma b_{2}+c_{1})\tau^{*}-b_{1}(u^{*})^{2}-\gamma c_{2}(v^{*})^{2}}
{(\delta u^{*}+\gamma v^{*})\lambda_{j}}>0.
\end{equation}
By substituting \eqref{const} into
\eqref{bif0}, one can verify that
\begin{equation}\label{bif}
d^{(j)}=
\dfrac{a_{2}^{\,2}b_{2}c_{2}}
{(b_{2}c_{1}-b_{1}c_{2})^{2}(\delta u^{*}+\gamma v^{*})\lambda_{j}}
D(a_{i}, b_{i}, c_{i},\gamma).
\end{equation}
It should be noted that 
$D(a_{i}, b_{i}, c_{i},\gamma)$ defined by \eqref{Ddef}
is independent of $j\in\mathbb{N}$.
Hence
$d^{(j)}>0$
if and only if
$D(a_{i}, b_{i}, c_{i},\gamma)>0$.
Furthermore,
if $D(a_{i}, b_{i}, c_{i},\gamma)>0$,
then 
\begin{equation}\label{djdeg}
0<d^{(j+1)}\le d^{(j)}\quad
\mbox{for any}\ j\in\mathbb{N}
\quad\mbox{and}\quad 
d^{(j)}=O(\lambda_{j}^{-1})
\quad\mbox{as}\ j\to\infty.
\end{equation}
Together with \eqref{sigma0} and \eqref{mjcut},
we can deduce that
the number $\sigma (d)$ of negative eigenvalues of $L(d)$
satisfies that
$$
\sigma (d)=
\begin{cases}
j\ \ \mbox{or}\ \ j+2\quad &\mbox{if}\ C<A<B
\quad\mbox{and}\quad d\in (d^{(j+1)}, d^{(j)}),\\
j+1\quad &\mbox{if}\ B<A<C
\quad\mbox{and}\quad d\in (d^{(j+1)}, d^{(j)}).
\end{cases}
$$
By virtue of \eqref{ind}, we establish the assertion of
Lemma \ref{indexlem}.
\end{proof}

\begin{proof}[Proof of Theorem \ref{exthm}]
Under assumptions of Theorem \ref{exthm},
we know from Theorem \ref{nonexthm} and Lemma \ref{bddlem} that 
if $d>\overline{d}$, then the following equation
of unknowns $(w,\tau)\in S_{\eta, M}$;
$$(w,\tau)- F(d,w,\tau)=0$$
has the unique solution $(w^{*},\tau^{*})$.
Therefore, the well-known property of the Leray-Schauder degree
implies that if $d>\overline{d}$, then
\begin{equation}\label{degld}
{\rm deg}\,(I-F(d,\,\sbt\,,\,\sbt\,), S_{\eta, M}, 0)=
{\rm ind}\,(I-F(d,\,\sbt\,,\,\sbt\,), (u^{*}, \tau^{*}))
=(-1)^{\sigma (d)}.
\end{equation}
It follows from \eqref{mjcut} and \eqref{djdeg} that
$\mu_{j}(d)>0$
for any $j\in\mathbb{N}$ if $d>d^{(1)}$.
Then, for any $d>d^{(1)}$,
the number $\sigma (d)$ of negative eigenvalues
of $I-L(d)$ is equal to
that of negatives of $\{\,\mu_{0}^{-}(d), \mu_{0}^{+}(d)\,\}$,
namely,
$\sigma (d)=\sigma_{0}(d)$.
It follows from \eqref{sigma0} and
\eqref{degld} that if $d>\overline{d}$, then
\begin{equation}\label{degpro}
{\rm deg}\,(I-F(d,\,\sbt\,,\,\sbt\,), S_{\eta, M}, 0)=
(-1)^{\sigma_{0}(d)}=
\begin{cases}
1\quad &\mbox{if}\ C<A<B,\\
-1\quad &\mbox{if}\ B<A<C.
\end{cases}
\end{equation}
Together with the homotopy invariance \eqref{hom}
of ${\rm deg}\,(I-F(d,\,\sbt\,,\,\sbt\,), S_{\eta, M}, 0)$,
we see that
\eqref{degpro} holds true for any $d\ge\varepsilon$.

Assume that $C<A<B$ in addition to \eqref{bifcond}.
We shall show that
\eqref{ISS} admits at least one nonconstant solution
when $d\in (d^{(j+1)}, d^{(j)})\cap [\varepsilon, \infty)$ and $j$ is odd.
Suppose for contradiction that
there is no nonconstant solution of \eqref{ISS}
for some $\hat{d}\in (d^{(j+1)}, d^{(j)})\cap [\varepsilon, \infty)$ 
with some odd $j$.
Then $(u^{*}, \tau^{*})$ is the only fixed point of 
$F(\hat{d},\,\sbt\,,\,\sbt\,)\,:\,S_{\eta, M}\to X$.
In this situation, we can use the index formula of the
Leray-Schauder degree to see
$$
{\rm deg}\,(I-F(\hat{d},\,\sbt\,,\,\sbt\,), S_{\eta, M}, 0)
=
{\rm ind}\,(I-F(\hat{d},\,\sbt\,,\,\sbt\,), (u^{*}, \tau^{*}))
=-1,
$$
where the last equality comes from Lemma \ref{indexlem}.
Obviously, this contradicts \eqref{degpro}.
By taking account for the arbitrary of $\varepsilon >0$,
we obtain the assertion in case (i) of Theorem \ref{exthm}.
Also in the other case $B<A<C$ with \eqref{bifcond},
a similar argument using Lemma \ref{indexlem}
proves the assertion in case (ii) of Theorem \ref{exthm}.
Then we complete the proof of Theorem \ref{exthm}.
\end{proof}

\section{One-dimensional analysis of the full cross-diffusion limit}
This section focuses on 
the global bifurcation structure of 
nonconstant solutions of \eqref{ISS} 
in the one-dimensional case $\Omega =(0,1)$.
Then we shall consider
the following nonlinear ordinary 
differential equation 
\begin{subequations}\label{IS1}
\begin{equation}\label{IS1-1} 
dw''+h(u,\tau)=0,\quad
u>0
\quad \mbox{in}\ (0,1),\quad \tau >0,
\end{equation}
subject to the homogeneous Neumann boundary condition
\begin{equation}\label{IS1-2}
w'(0)=w'(1)=0
\end{equation}
with the integral constraint
\begin{equation}\label{IS1-3}
\displaystyle\int^{1}_{0}
f\biggl(u,\dfrac{\tau}{u}\biggr)=0,
\end{equation}
\end{subequations}
where $u=u(w,\tau )$ is defined by \eqref{uvdef}
and $h(u,\tau )$ is the nonlinear term given as \eqref{hdef}.
It follows from \eqref{intg} that
\begin{equation}\label{gint2}
\displaystyle\int^{1}_{0}
g\biggl(u,\dfrac{\tau}{u}\biggr)=0.
\end{equation}
In this section, the prime symbol represents the derivative by $x$.
It will be shown that,
for each fixed small $\tau>0$,
there exist three zeros of
$h(u,\tau )$
on $\{u>0\}$:
\begin{lem}\label{hlem}
If $\tau>0$ is sufficiently small,
then
$h(u,\tau )$ 
$(u>0)$ possesses three zeros 
$
0<
z_{1}(\tau )<
z_{2}(\tau )<
z_{3}(\tau )
$
such that
$$h(u,\tau )
\begin{cases}
>0\quad &\mbox{for}\ u\in (0,z_{1}(\tau ))\cup (z_{2}(\tau ), z_{3}(\tau )),\\
<0\quad &\mbox{for}\ u\in (z_{1}(\tau ), z_{2}(\tau ))
\cup (z_{3}(\tau ), \infty)
\end{cases}
$$
and
$$
\lim_{\tau\searrow 0}\dfrac{z_{1}(\tau )}{\tau }=\dfrac{c_{2}}{a_{2}},\quad
\lim_{\tau\searrow 0}\dfrac{z_{2}(\tau )}{\sqrt{\tau }}
=\sqrt{\dfrac{\gamma a_{2}}{a_{1}}},\quad
\lim_{\tau\searrow 0}z_{3}(\tau )=\dfrac{a_{1}}{b_{1}}.
$$
\end{lem} 

\begin{proof}
Obviously, the fundamental theorem of algebra ensures that
$h(u,\tau )$
has at most three zeros on $\{u>0\}$.
By a straightforward calculation, one can verify
$$
\lim_{\tau\searrow 0} h(\kappa_{1}\tau, \tau)=
-\dfrac{\gamma }{\kappa_{1}}\biggl(a_{2}-\dfrac{c_{2}}{\kappa_{1} }\biggr)
\begin{cases}
>0\quad &\mbox{if}\ \kappa_{1}\in (0,c_{2}/a_{2}),\\
=0\quad &\mbox{if}\ \kappa_{1}=c_{2}/a_{2},\\
<0\quad &\mbox{if}\ \kappa_{1}\in (c_{2}/a_{2},\infty)
\end{cases}
$$ and
$$
\lim_{\tau\searrow 0} \frac{h(\kappa_{2}\sqrt{\tau}, \tau)}{\sqrt{\tau }}=
\kappa_{2}a_{1}-\dfrac{\gamma a_{2}}{\kappa_{2}}
\begin{cases}
<0\quad &\mbox{if}\ \kappa_{2}\in (0,\sqrt{\gamma a_{2}/a_{1}}),\\
=0\quad &\mbox{if}\ \kappa_{2}=\sqrt{\gamma a_{2}/a_{1}},\\
>0\quad &\mbox{if}\ \kappa_{2}\in (\sqrt{\gamma a_{2}/a_{1}},\infty)
\end{cases}
$$ and
$$
\lim_{\tau\searrow 0} h(u,\tau )=
u(a_{1}-b_{1}u)
\begin{cases}
>0\quad &\mbox{if}\ u\in (0,a_{1}/b_{1}),\\
=0\quad &\mbox{if}\ u=a_{1}/b_{1},\\
<0\quad &\mbox{if}\ u\in (a_{1}/b_{1},\infty).
\end{cases}
$$ 
Therefore, we obtain the required assertion.
\end{proof}
We introduce the set $\mathcal{T}=\mathcal{T}(a_{i}, b_{i}, c_{i}, \gamma )$
such as
$$
\mathcal{T}:=
\{\,\tau\in (0, \overline{\tau}]\,:\,
h(u,\tau)\ \mbox{has three zeros}\ 
0<z_{1}(\tau)<z_{2}(\tau)<z_{3}(\tau )\,\},
$$
where $\overline{\tau}$ is obtained in Theorem \ref{nonexthm}.
Hence $\mathcal{T}$ is not empty because any small $\tau>0$
belongs to $\mathcal{T}$ by Lemma \ref{hlem}.
Our strategy of analysis of \eqref{IS1} is as follows:
First we obtain solutions of the Neumann problem
\eqref{IS1-1}-\eqref{IS1-2}
without the integral constraint \eqref{IS1-3}.
Next we construct the set of solutions of 
\eqref{IS1} by choosing
functions satisfying 
\eqref{IS1-3} in the set of solutions of
\eqref{IS1-1}-\eqref{IS1-2}.
To this end, we first obtain the
following existence of solutions of \eqref{IS1-1}-\eqref{IS1-2}:
\begin{prop}\label{timemapprop}
Suppose that $\tau\in\mathcal{T}$ and $h_{u}(z_{2}(\tau), \tau )>0$.
For each $j\in\mathbb{N}$,
if 
$$
0<d<\dfrac{h_{u}(z_{2}(\tau ), \tau)}
{(\delta +\frac{\gamma\tau}{z_{2}(\tau )})(j \pi)^{2}},$$
then \eqref{IS1-1}-\eqref{IS1-2}
admits at least two solutions
$w^{+}_{j}(x; d, \tau)$
and
$w^{-}_{j}(x; d, \tau)$
which satisfy
\begin{equation}\label{osc+}
(-1)^{i-1}(w_{j}^{+})'(x; d, \tau )>0\quad\mbox{for any}
\ x\in\biggl(\dfrac{i-1}{j}, \dfrac{i}{j}\biggr),
\quad (i=1,2,\ldots,j)
\end{equation}
and
\begin{equation}\label{osc-}
(-1)^{i-1}(w_{j}^{-})'(x; d, \tau )<0\quad\mbox{for any}
\ x\in\biggl(\dfrac{i-1}{j}, \dfrac{i}{j}\biggr),
\quad (i=1,2,\ldots,j).
\end{equation}
\end{prop}

\begin{proof}
In order to find solutions 
of the Neumann problem \eqref{IS1-1}-\eqref{IS1-2}
by the shooting method,
we consider the associated initial-value problem
\begin{equation}\label{ini}
\begin{cases}
dw''+h(u,\tau )=0,
\quad x>0,\\
w(0)=m>0,\quad
w'(0)=0,
\end{cases}
\end{equation}
where $u$ is defined by \eqref{uvdef}.
In the rising part of the proof,
following the standard shooting method,
we multiply the differential equation of \eqref{ini}
by $w'$ as follows:
\begin{equation}\label{ode}
dw'w''+h(u,\tau)w'=0.
\end{equation}
Noting here $w=\delta u -\gamma\tau/ u$,
we substitute
\begin{equation}\label{wp}
w'=\biggl(\delta +\dfrac{\gamma\tau}{u^{2}}\biggr) u'
\end{equation} 
into the latter $w$ in \eqref{ode} to get
$$
\biggl(
\dfrac{d}{2}w'(x)^{2}+H(u(x),\tau )\biggr)'=0,
$$
where
\begin{equation}\label{Hdef}
H(u,\tau )=\displaystyle\int^{u}_{z_{2}(\tau )}
h(s,\tau)\biggl(
\delta +\dfrac{\gamma \tau}{s^{2}}\biggr)\,{ds}.
\end{equation}
Obviously,
for each fixed $\tau\in\mathcal{T}$,
the function $H(u,\tau )$ 
$(u>0)$ attains local maximums
at $u=z_{1}(\tau ),\,z_{3}(\tau )$ and a local minimum at $u=z_{2}(\tau )$.
Then we obtain
$$
\dfrac{d}{2}w'(x)^{2}+H(u(x),\tau )=
H(m,\tau )
$$
for any $x$ as long as the solution of \eqref{ini} exists.
Hence any solution $w$ with monotone increasing for 
small $x>0$ satisfies
$$
w'(x)=
\sqrt{\dfrac{2}{d}}
\sqrt{H(m,\tau )-H(u,\tau)}
$$
With \eqref{wp}, one can see
$$
u'(x)=
\sqrt{\dfrac{2}{d}}
\dfrac{\sqrt{H(m, \tau)-H(u, \tau)}}
{\delta+\frac{\gamma \tau}{u^{2}}}.
$$
Therefore, derivatives of inverse functions give
\begin{equation}\label{dxdu}
\dfrac{dx}{du}=
\sqrt{\dfrac{d}{2}}
\biggl(
\dfrac
{\delta}
{\sqrt{H(m, \tau)-H(u, \tau)}}
+
\dfrac
{\gamma\tau}
{u^{2}\sqrt{H(m, \tau)-H(u, \tau)}}
\biggr)
\end{equation}
for any $x$
as long as the solution $w(x)$ of \eqref{ini}
fulfills $u'(x)>0$.

If necessary, we denote by $w(x,m)$ and $u(x,m)$
the solution $w$ of \eqref{ini} and $u$ defined
by \eqref{uvdef} in order to specify the dependence on $m$.
By \eqref{wp},
increase/decrease of $w(x,m)$ and $u(x,m)$ matches.

The following shooting argument using the 
$(u,u')$ phase plane
will be divided into two cases (i) and (ii);
$$
{\rm (i)}\
H(z_{1}(\tau ), \tau )\le H(z_{3}(\tau ), \tau );
\quad
{\rm (ii)}\
H(z_{1}(\tau ), \tau )> H(z_{3}(\tau ), \tau ).
$$

In case (i),
for any $m\in (z_{1}(\tau ), z_{2}(\tau ))$,
there exists 
$M(m, \tau )\in (z_{2}(\tau ), z_{3}(\tau ))$ such that
$$
H(m, \tau )=H(M(m, \tau ), \tau )
\quad\mbox{and}\quad
M(m, \tau )\searrow z_{2}(\tau ) 
\ \mbox{as}\ 
m\nearrow z_{2}(\tau ).
$$
Here we set
\begin{equation}\label{Xdef0}
X(m,\tau):=\sup\{\,\widetilde{x}>0
\,:\,u'(x,m)>0\ \mbox{for any}\ x\in (0,\widetilde{x})\,\}
\end{equation}
for $\tau \in \mathcal{T}$.
From \eqref{dxdu},
a standard analysis using the $(u,u')$ phase plane
enables us to see that
$X(m,\tau )$ is well-defined and finite if and only if
$m\in (z_{1}(\tau ), z_{2}(\tau ))$.
In this case,
$u(x,m)$ is monotone increasing for $x\in (0,X(m,\tau ))$
with $u(X(m,\tau ),\tau )=M(m, \tau )$.
Actually,
integrating \eqref{dxdu} by $u$ over $(m, M(m, \tau ))$, we get
\begin{equation}\label{Xdef}
X(m.\tau )=
\sqrt{\dfrac{d}{2}}\,
\{\,\delta I(m,\tau )+\gamma\tau J(m,\tau )\,\}
\end{equation}
for any $(m,\tau )\in (z_{1}(\tau ), z_{2}(\tau ))\times\mathcal{T}$,
where
$$
I(m,\tau ):=
\displaystyle\int^{M(m, \tau )}_{m}
\dfrac{du}{\sqrt{H(m,\tau )-H(u,\tau)}}
$$
and
$$
J(m,\tau ):=
\displaystyle\int^{M(m, \tau )}_{m}
\dfrac{du}{u^{2}\sqrt{H(m,\tau )-H(u,\tau)}}.
$$
By the change of variables 
$u=m+(M(m, \tau )-m)\theta$,
one can see
\begin{equation}\label{I2}
I(m,\tau)=
(M(m, \tau )-m)\displaystyle\int^{1}_{0}
\dfrac{d\theta}
{\sqrt{H(m,\tau )-
H(m+(M(m, \tau )-m)\theta, \tau)}}
\end{equation}
and
\begin{equation}\label{J2}
J(m,\tau)=
(M(m, \tau )-m)\displaystyle\int^{1}_{0}
\dfrac{d\theta}
{\{\,m+(M(m, \tau )-m)\theta\,\}^{2}\sqrt{H(m,\tau )-
H(m+(M(m, \tau )-m)\theta, \tau)}}.
\end{equation}

In order to derive the asymptotic behavior of 
$I(m,\tau)$ and $J(m,\tau )$ as $m\nearrow z_{2}(\tau )$,
we expand $H(m,\tau )$ and $H(m+(M(m, \tau )-m)\theta , \tau )$
into Taylor's series around $z_{2}(\tau )$ as follows:
\begin{equation}
\begin{split}
&H(m,\tau )-H(m+(M(m, \tau )-m)\theta , \tau ) \\
=&
H(m,\tau )-H(z_{2}(\tau ), \tau )
-\{\,H(m+(M(m, \tau )-m)\theta , \tau )-H(z_{2}(\tau ), \tau )\,\}\\
=&
H_{u}(z_{2}(\tau ), \tau )(m-z_{2}(\tau ))+
\dfrac{H_{uu}(z_{2}(\tau ), \tau )}{2}(m-z_{2}(\tau ))^{2}+
o((m-z_{2}(\tau ))^{2})\\
&-H_{u}(z_{2}(\tau ), \tau )
\{\,m+(M(m, \tau )-m)\theta -z_{2}(\tau )\,\}
-\dfrac{H_{uu}(z_{2}(\tau ), \tau )}{2}
\{\,m+(M(m, \tau )-m)\theta -z_{2}(\tau )\,\}^{2}\\
&+o(\{\,m+(M(m, \tau )-m)\theta -z_{2}(\tau )\,\}^{2}).
\end{split}
\nonumber
\end{equation}
Here we recall \eqref{Hdef} to note
$$
H_{u}(z_{2}(\tau ),\tau )=0\quad \mbox{and}\quad
H_{uu}(z_{2}(\tau ), \tau )=h_{u}(z_{2}(\tau ),\tau )\biggl(
\delta +\dfrac{\gamma \tau}{z_{2}(\tau )^{2}}\biggr).
$$
Then it follows that
\begin{equation}\label{h}
\begin{split}
&H(m,\tau )-H(m+(M(m, \tau )-m)\theta , \tau ) \\
=&
\dfrac{h_{u}(z_{2}(\tau ),\tau )}{2}\biggl(
\delta +\dfrac{\gamma \tau}{z_{2}(\tau )^{2}}\biggr)
(M(m, \tau )-m)^{2}
\biggl(
\dfrac{2(z_{2}(\tau )-m)}{M(m, \tau )-m}-\theta\biggr)\theta\\
&
+o((m-z_{2}(\tau ))^{2})
+o((m+(M(m, \tau )-m)\theta -z_{2}(\tau ))^{2})
\end{split}
\end{equation}
as $m\nearrow z_{2}(\tau )$.
Here we shall show 
\begin{equation}\label{lr}
\lim_{m\nearrow z_{2}(\tau )}
\dfrac{2(z_{2}(\tau )-m)}{M(m, \tau )-m}=1.
\end{equation}
Differentiating $H(m,\tau )=H(M(m,\tau ), \tau )$ by $m$, we see
$H_{u}(m,\tau )=
H_{u}(M(m,\tau ), \tau )M_{m}(m,\tau )$
for any $(m,\tau )\in (z_{1}(\tau ), z_{2}(\tau ))\times
\mathcal{T}$.
By H\^opital's rule, it is easy to check that
$$
\lim_{m\nearrow z_{2}(\tau )}
M_{m}(m,\tau )=
\lim_{m\nearrow z_{2}(\tau )}
\dfrac{H_{u}(m,\tau )}
{H_{u}(M(m,\tau ), \tau )}
=
\lim_{m\nearrow z_{2}(\tau )}
\dfrac{h_{u}(m,\tau )}
{h_{u}(M(m,\tau ), \tau )M_{m}(m,\tau )}.
$$
Therefore, we obtain
$$
\lim_{m\nearrow z_{2}(\tau )}M_{m}(m, \tau )^{2}
=
\lim_{m\nearrow z_{2}(\tau )}
\dfrac{h_{u}(m,\tau )}
{h_{u}(M(m,\tau ), \tau )}=1
$$
because $h_{u}(z_{2}(\tau ), \tau )>0$.
By the fact that
$M(m,\tau )$ is monotone decreasing 
for $m\in (z_{1}(\tau ), z_{2}(\tau ))$
with each fixed $\tau\in\mathcal{T}$,
we know that
$
M_{m}(m, \tau )
\to -1
$
as 
$m\nearrow z_{2}(\tau )$.
Therefore, we obtain \eqref{lr} by using H\^opital's rule
as follows:
$$
\lim_{m\nearrow z_{2}(\tau )}
\dfrac{2(z_{2}(\tau )-m)}{M(m, \tau )-m}=
\lim_{m\nearrow z_{2}(\tau )}
\dfrac{-2}{M_{m}(m, \tau )-1}=1.
$$
Substituting \eqref{h} into \eqref{I2}
and \eqref{J2}, and then,
setting $m\nearrow z_{2}(\tau )$,
we know from \eqref{lr} that
$$
\lim_{m\nearrow z_{2}(\tau )}
I(m,\tau )=
\sqrt{
\dfrac{2}{h_{u}(z_{2}(\tau ), \tau )
(\delta +\frac{\gamma \tau }{z_{2}(\tau )^{2}})}}
\displaystyle\int^{1}_{0}
\frac{d\theta}{\sqrt{\theta (1-\theta )}}
=\sqrt{\dfrac{2}{h_{u}(z_{2}(\tau ), \tau )
(\delta +\frac{\gamma \tau }{z_{2}(\tau )^{2}})}}\,\pi
$$
and
$$
\lim_{m\nearrow z_{2}(\tau )}
J(m,\tau )
=\dfrac{1}{z_{2}(\tau )^{2}}\sqrt{\dfrac{2}{h_{u}(z_{2}(\tau ), \tau )
(\delta +\frac{\gamma \tau }{z_{2}(\tau )^{2}})}}\,\pi.
$$
Consequently, 
we set $m\nearrow z_{2}(\tau )$ in \eqref{Xdef} to get
\begin{equation}  
\lim_{m\nearrow z_{2}(\tau )}X(m,\tau )
=\sqrt{
\dfrac{d(\delta +\frac{\gamma \tau}{z_{2}(\tau )^{2}})}
{h_{u}(z_{2}(\tau ), \tau )}}\,\pi.
\end{equation}

In case (i),
for the derivation of the asymptotic behavior of $X(m,\tau )$ as
$m\searrow z_{1}(\tau )$,
we use the Taylor expansion of
$\theta\mapsto H(m+(M(m,\tau)-m)\theta, \tau)$
around $\theta =0$ to observe
\begin{equation}
\begin{split}
&H(m,\tau )-H(m+(M(m, \tau)-m)\theta, \tau )\\
=&
-h(m,\tau )\biggl(\delta+\dfrac{\gamma\tau}{m^{2}}\biggr)
(M(m,\tau )-m)\theta\\
&-
\dfrac{1}{2}
\biggl\{\,
h_{u}(m,\tau)\biggl(\delta+\dfrac{\gamma\tau}{m^{2}}\biggr)
-h(m,\tau)\dfrac{2\gamma\tau}{m^{3}}\,\biggr\}
(M(m,\tau )-m)^{2}\theta^{2}
+o((M(m,\tau )-m)^{2}\theta^{2})
\end{split}
\nonumber
\end{equation}
for each $m\in (z_{1}(\tau ), z_{2}(\tau ))$
as $\theta\searrow 0$.
Here we assume
$
h_{u}(z_{1}(\tau ),\tau )<0$.
Then for any $\tau\in\mathcal{T}$,
there exist
$k_{i}(m,\tau )>0$ $(i=1,2)$
with
$\lim_{m\searrow z_{1}(\tau )}k_{1}(m,\tau )=0$
and
$\lim_{m\searrow z_{1}(\tau )}k_{2}(m,\tau )>0$
such that
if $\theta>0$ is sufficiently small and
$m\in (z_{1}(\tau ), z_{2}(\tau ))$
is sufficiently close to $z_{1}(\tau )$,
then
$$H(m,\tau )-H(m+(M(m, \tau)-m)\theta, \tau )
\le k_{1}(m,\tau )\theta +k_{2}(m,\tau )\theta^{2}.$$
Therefore, there exists a small $\varepsilon >0$ such that
$$
I(m,\tau )\ge 
(M(m, \tau)-m)\displaystyle\int^{\varepsilon}_{0}
\dfrac{d\theta}
{\sqrt{k_{1}(m,\tau )\theta+
k_{2}(m,\tau )\theta^{2}}}
$$
and
$$
J(m,\tau )\ge 
\dfrac{M(m, \tau )-m}{z_{2}(\tau )^{2}}\displaystyle\int^{\varepsilon}_{0}
\dfrac{d\theta}
{\sqrt{k_{1}(m,\tau )\theta+
k_{2}(m,\tau )\theta^{2}}}
$$
if $m\in (z_{1}(\tau ), z_{2}(\tau ))$ is sufficiently close to $z_{1}(\tau )$.
Consequently, we know from \eqref{Xdef} that,
in case (i),
\begin{equation}\label{z1infty}
\lim_{m\searrow z_{1}(\tau )}X(m,\tau )=\infty.
\end{equation}
It is easy to check that \eqref{z1infty} holds true even when
$h_{u}(z_{1}(\tau ), \tau )=0$ by observing the higher order
expansion of $\theta\mapsto
H(m+(M(m,\tau)-m)\theta, \tau )$.

Next we consider the other case (ii);
$H(z_{1}(\tau ),\tau )>H(z_{3}(\tau ), \tau ))$.
For such $\tau\in\mathcal{T}$,
there exists a unique $\underline{m}(\tau )\in (z_{1}(\tau ), z_{2}(\tau ))$
such that
$H(\underline{m}(\tau ), \tau )=H(z_{3}(\tau ), \tau )$.
In case (ii),
$X(m,\tau )$ in \eqref{Xdef0} is well-defined and finite
if and only if $m\in (\underline{m}(\tau ), z_{2}(\tau ))$,
and moreover, 
it is represented as
\eqref{Xdef}.
By a similar manner as in case (i), one can verify that
$$
\lim_{m\searrow \underline{m}(\tau )}X(m,\tau )=\infty
\quad\mbox{and}\quad
\lim_{m\nearrow z_{2}(\tau )}X(m,\tau )=
\sqrt{
\dfrac{d(\delta +\frac{\gamma \tau}{z_{2}(\tau )^{2}})}
{h_{u}(z_{2}(\tau ), \tau )}}\,\pi
$$
in case (ii).

Therefore,
in both cases (i) and (ii),
the algebraic equation $X(m,\tau)=1/j$ admits at least one root
$m=m_{j}(\tau )\in 
(z_{1}(\tau ), z_{2}(\tau ))$ or
$(\underline{m}(\tau ), z_{2}(\tau ))$
provided
$$
\lim_{m\nearrow z_{2}(\tau )}X(m,\tau )=
\sqrt{
\dfrac{d(\delta +\frac{\gamma \tau}{z_{2}(\tau )^{2}})}
{h_{u}(z_{2}(\tau ), \tau )}}\,\pi<\dfrac{1}{j},
\quad\mbox{that is,}\quad
0<d<\dfrac{h_{u}(z_{2}(\tau ), \tau)}
{(\delta +\frac{\gamma\tau}{z_{2}(\tau )^{2}})(j \pi)^{2}}.
$$
For such $d$,
the solution $w(x,m_{j}(\tau ))$ of 
the initial-value problem \eqref{ini}
satisfies 
$w'(x, m_{j}(\tau ))>0$ for $x\in (0,1/j)$;
$w'(1/j, m_{j}(\tau ))=0$;
$w'(x, m_{j}(\tau ))<0$ for $x\in (1/j, 2/j)$;
$w(2/j, m_{j}(\tau ))=w(0, m_{j}(\tau))$;
$w'(2/j, m_{j}(\tau ))=0$,
and moreover,
oscillates periodically for $x>0$.
Then 
$$w^{+}_{j}(x;d, \tau):=w(x,m_{j}(\tau ))$$
becomes a solution of the Neumann problem 
\eqref{IS1-1}-\eqref{IS1-2} and it satisfies \eqref{osc+}.
Furthermore,
$w^{-}_{j}(x;d, \tau):=w(x+1/j,m_{j}(\tau ))$
is also a solution of \eqref{IS1-1}-\eqref{IS1-2} 
and satisfies \eqref{osc-}.
The proof of Proposition \ref{timemapprop} is complete.
\end{proof}

\begin{rem}
It is possible to prove that if
$h_{u}(z_{2}(\tau ), \tau )=0$,
then $\lim_{m\nearrow z_{2}(\tau )}X(m,\tau )=\infty$.
Together with $\lim_{m\searrow z_{1}(\tau )}X(m,\tau )=\infty$
or $\lim_{m\searrow \underline{m}(\tau )}X(m,\tau )=\infty$,
for each $j\in\mathbb{N}$,
there exists $\hat{d}^{(j)}>0$ such that
if $d\in (0, \hat{d}^{(j)})$, then
\eqref{IS1-1}-\eqref{IS1-2} has at least four solutions
$\overline{w}^{+}_{j}(x;d,\tau)$,
$\underline{w}^{+}_{j}(x;d,\tau)$,
$\overline{w}^{-}_{j}(x;d,\tau)$ and
$\underline{w}^{-}_{j}(x;d,\tau)$,
where
$\overline{w}^{+}_{j}(x;d,\tau)$ and
$\underline{w}^{+}_{j}(x;d,\tau)$
satisfy \eqref{osc+};
$\overline{w}^{-}_{j}(x;d,\tau)$ and
$\underline{w}^{-}_{j}(x;d,\tau)$
satisfy \eqref{osc-}.
\end{rem}

\begin{rem}\label{taubarrem}
If $\tau\not\in\mathcal{T}$, then
\eqref{IS1} does not admit any nonconstant solution.
Actually,
a standard phase plane analysis implies that
any solution of \eqref{ini} cannot satisfy
$w'(1)=0$ in case the number of zeros of
$h(u,\tau)$ is less than three.
\end{rem}

Concerning the Neumann problem \eqref{IS1-1}-\eqref{IS1-2}
(without \eqref{IS1-3}),
we discuss the singular limit as $d\searrow 0$ 
of solutions 
obtained in Proposition \ref{timemapprop}.
It is possible to verify that,
for $n\ge 2$,
each $w^{\pm}_{n}(x;d, \tau)$ can be constructed by
connecting suitable rescaled or reflected pieces of 
$w^{+}_{1}(x;d, \tau)$.
Then we study the singular limit as $d\searrow 0$ of 
$w_{1}^{+}(x;d, \tau)$. 
Hereafter we use another notation 
$\widetilde{h}(w,\tau )$
of the nonlinear term of \eqref{IS1-1}
by substituting $u(w,\tau )$ defined by \eqref{uvdef} 
into $h(u,\tau )$ such as 
$
\widetilde{h}(w,\tau)
:=h(u(w,\tau ), \tau )$.
Then \eqref{IS1-1}-\eqref{IS1-2} can be represented as
\begin{equation}\label{IS1p}
\begin{cases}
dw''+\widetilde{h}(w,\tau )=0\quad\mbox{in}\ (0,1),\quad\tau >0,\\
w'(0)=w'(1)=0.
\end{cases}
\end{equation}
In view of Lemma \ref{hlem}, we define $\xi_{i}(\tau )$
$(i=1, 2, 3)$ by
$$
z_{i}(\tau )=u(\xi_{i}(\tau ), \tau),\quad
\mbox{conversely},\quad
\xi_{i}(\tau ):= \delta z_{i}(\tau )-\dfrac{\gamma\tau}{z_{i}(\tau )}.
$$
From the monotone increasing relation of $w\mapsto u(w,\tau )$
by \eqref{wp},
we know from Lemma \ref{hlem} that if $\tau\in\mathcal{T}$,
then
$$
\widetilde{h}(w,\tau )
\begin{cases}
>0\quad &\mbox{for}\ w\in (-\infty,\xi_{1}(\tau ))
\cup (\xi_{2}(\tau ), \xi_{3}(\tau )),\\
<0\quad &\mbox{for}\ w\in (\xi_{1}(\tau ), \xi_{2}(\tau ))
\cup (\xi_{3}(\tau ), \infty).
\end{cases}
$$
For such a bistable nonlinear term,
we set
\begin{equation}\label{Htdef}
\widetilde{H}(w,\tau ):=
\int^{w}_{\xi_{2}(\tau )}\widetilde{h}(s,\tau )\,ds.
\end{equation}
Concerning the Neumann problem of ordinary 
differential equations with a class of bistable
nonlinearities such as $\widetilde{h}(w,\tau )$,
it is well known that the singular limit of solutions
as $d\searrow 0$ crucially depends on the sign
of $\widetilde{H}(\xi_{3}(\tau ), \tau )-
\widetilde{H}(\xi_{1}(\tau ), \tau )$ as follows
(see e.g., 
Nishiura \cite[Lemma 3.1]{NiS},
Shi \cite[Proposition 2.6]{Sh}):
\begin{enumerate}[(i)]
\item
If $\widetilde{H}(\xi_{1}(\tau ), \tau )<\widetilde{H}(\xi_{3}(\tau ), \tau )$, then
$$
\lim_{d\searrow 0}w_{1}^{+}(x;d, \tau)=
\begin{cases}
\xi_{1}(\tau )\quad &\mbox{for}\ x\in [0,1),\\
\widetilde{\eta }(\tau ) &\mbox{for}\ x=1,
\end{cases}
$$
where $\widetilde{\eta}(\tau )\in (\xi_{2}(\tau ), \xi_{3}(\tau ))$
is defined by 
$\int^{\widetilde{\eta}(\tau )}_{\xi_{1}(\tau )}\widetilde{h}(s,\tau )ds=0$.
\item
If $\widetilde{H}(\xi_{1}(\tau ), \tau)=
\widetilde{H}(\xi_{3}(\tau ), \tau )$, then
\begin{equation}
\lim_{d\searrow 0}w_{1}^{+}(x;d, \tau)=
\begin{cases}
\xi_{1}(\tau )\quad &\mbox{for}\ x\in [0,1/2),\\
(\xi_{1}(\tau )+\xi_{3}(\tau ))/2\quad &\mbox{for}\ x=1/2,\\
\xi_{3}(\tau ) &\mbox{for}\ x\in (1/2, 1].
\end{cases}
\nonumber
\end{equation}
\item
If $\widetilde{H}(\xi_{1}(\tau ), \tau )>\widetilde{H}(\xi_{3}(\tau ), \tau )$, then
\begin{equation}\label{iii}
\lim_{d\searrow 0}w_{1}^{+}(x;d, \tau)=
\begin{cases}
\widetilde{\zeta }(\tau ) &\mbox{for}\ x=0,\\
\xi_{3}(\tau )\quad &\mbox{for}\ x\in (0,1],
\end{cases}
\nonumber
\end{equation}
where $\widetilde{\zeta} (\tau )\in (\xi_{1}(\tau ),\xi_{2}(\tau ))$ is defined by 
$\int^{\xi_{3}(\tau )}_{\widetilde{\zeta}(\tau )}\widetilde{h}(s,\tau )ds=0$.
\end{enumerate}
Here we note that
the change of variables
$w=\delta u-\gamma \tau/u$
links \eqref{Hdef} with \eqref{Htdef}
in the sense of
$
\widetilde{H}(w,\tau )=H(u,\tau)$.
Then by the change of variables,
the above (i)-(iii) give the following singular limiting behavior of 
\begin{equation}\label{u+def}
u^{+}_{1}(x; d, \tau):=
\dfrac{\sqrt{w^{+}_{1}(x;d, \tau)^{2}+4\gamma\delta\tau}+w^{+}_{1}(x; d, \tau)}{2\delta }
\end{equation}
as $d\searrow 0$:
\begin{lem}\label{singlem}
Suppose that $\tau\in\mathcal{T}$.
The function $u^{+}_{1}(x; d, \tau)$
satisfies either one of the following (i)-(iii)
depending on the sign of 
$H(z_{3}(\tau ),\tau)-H_{1}(z_{1}(\tau ), \tau )$:
\begin{enumerate}[{\rm (i)}]
\item
If $H(z_{1}(\tau ), \tau )<H(z_{3}(\tau ), \tau )$, then
\begin{equation}
\lim_{d\searrow 0}u_{1}^{+}(x;d, \tau)=
\begin{cases}
z_{1}(\tau )\quad &\mbox{for}\ x\in [0,1),\\
\eta (\tau ) &\mbox{for}\ x=1,
\end{cases}
\nonumber
\end{equation}
where $\eta(\tau )\in (z_{2}(\tau ), z_{3}(\tau ))$
is defined by 
$\int^{\eta(\tau )}_{z_{1}(\tau )}h(s,\tau )
(\delta +\frac{\gamma \tau}{s^{2}})ds=0$.
\item
If $H(z_{1}(\tau ), \tau)=H(z_{3}(\tau ), \tau )$, then
\begin{equation}
\lim_{d\searrow 0}u_{1}^{+}(x;d, \tau)=
\begin{cases}
z_{1}(\tau )\quad &\mbox{for}\ x\in [0,1/2),\\
(z_{1}(\tau )+z_{3}(\tau ))/2\quad &\mbox{for}\ x=1/2,\\
z_{3}(\tau ) &\mbox{for}\ x\in (1/2, 1].
\end{cases}
\nonumber
\end{equation}
\item
If $H(z_{1}(\tau ), \tau )>H(z_{3}(\tau ), \tau )$, then
\begin{equation}
\lim_{d\searrow 0}u_{1}^{+}(x;d, \tau)=
\begin{cases}
\zeta (\tau ) &\mbox{for}\ x=0,\\
z_{3}(\tau )\quad &\mbox{for}\ x\in (0,1],
\end{cases}
\nonumber
\end{equation}
where $\zeta (\tau )\in (z_{1}(\tau ),z_{2}(\tau ))$ is defined by 
$\int^{z_{3}(\tau )}_{\zeta(\tau )}h(s,\tau )
(\delta +\frac{\gamma \tau}{s^{2}})ds=0$.
\end{enumerate}
\end{lem}

Our next task is to construct the set of nonconstant
solutions of \eqref{IS1}
by choosing functions in the set of solutions
of \eqref{IS1-1}-\eqref{IS1-2} to match \eqref{IS1-3}.
The following lemma will be useful to catch up with
the global bifurcation branch of nonconstant solutions of
\eqref{IS1}.
\begin{lem}\label{fsflem}
Suppose that $\tau >0$ is 
sufficiently small.
Then the following properties hold:
\begin{enumerate}[{\rm (i)}]
\item
If $B<A<C$,
then
$$
f\biggl(z_{1}(\tau ), \dfrac{\tau}{z_{1}(\tau )}\biggr)<0,
\quad
f\biggl(z_{2}(\tau ), \dfrac{\tau}{z_{2}(\tau )}\biggr)>0
\quad\mbox{and}\quad
f\biggl(z_{3}(\tau ), \dfrac{\tau}{z_{3}(\tau )}\biggr)<0.$$
\item
If $C<A<B$,
then
$$
f\biggl(z_{1}(\tau ), \dfrac{\tau}{z_{1}(\tau )}\biggr)>0,
\quad
f\biggl(z_{2}(\tau ), \dfrac{\tau}{z_{2}(\tau )}\biggr)>0
\quad\mbox{and}\quad
f\biggl(z_{3}(\tau ), \dfrac{\tau}{z_{3}(\tau )}\biggr)>0.$$
\end{enumerate}
\end{lem}

\begin{proof}
In view of \eqref{hdef},
we note that
three zeros of $h(u,\tau )$ 
are corresponding to
three intersections of
$u\mapsto f(u,\tau/u)$ and $u\mapsto g(u,\tau/u)$
on $\{\,u>0\,\}$.
Thus we summarize the profiles of
$$f\biggl(u, \dfrac{\tau}{u}\biggr)=u(a_{1}-b_{1}u)-c_{1}\tau
\quad
\mbox{and}
\quad
g\biggl(u, \dfrac{\tau}{u}\biggr)=
\dfrac{\tau}{u}\biggl(
a_{2}-\dfrac{c_{2}\tau}{u}\biggr)-b_{2}\tau
$$
for $u>0$ with each fixed small $\tau >0$.
It is noted that
\begin{equation}
\begin{split}
&\lim_{u\searrow 0}f\biggl(u, \dfrac{\tau}{u}\biggr)=-c_{1}\tau,\quad
\lim_{u\to\infty}f\biggl(u, \dfrac{\tau}{u}\biggr)=-\infty,\quad
\\
&\lim_{u\searrow 0}g\biggl(u, \dfrac{\tau}{u}\biggr)=-\infty,\quad\ \,
\lim_{u\to\infty}g\biggl(u, \dfrac{\tau}{u}\biggr)=-b_{2}\tau.
\end{split}
\nonumber
\end{equation}
Clearly, $u\mapsto f(u, \tau/u )$ has two zeros
$$
0<Z_{1}(f,\tau ):=\dfrac{a_{1}-\sqrt{a_{1}^{2}-4b_{1}c_{1}\tau }}{2b_{1}}
<
Z_{2}(f,\tau ):=\dfrac{a_{1}+\sqrt{a_{1}^{2}-4b_{1}c_{1}\tau }}{2b_{1}}
$$
if $\tau\in (0,a_{1}^{\,2}/4b_{1}c_{1})$,
whereas
$u\mapsto g(u, \tau/u )$ has two zeros
$$
0<Z_{1}(g,\tau ):=\dfrac{a_{2}-\sqrt{a_{2}^{2}-4b_{2}c_{2}\tau }}{2b_{2}}
<
Z_{2}(g,\tau ):=\dfrac{a_{2}+\sqrt{a_{2}^{2}-4b_{2}c_{2}\tau }}{2b_{2}}
$$
if $\tau\in (0,a_{2}^{\,2}/4b_{2}c_{2})$.
It follows from
$$
\lim_{\tau\searrow 0}\dfrac{Z_{1}(f,\tau )}{\tau }=\dfrac{c_{1}}{a_{1}}
\quad\mbox{and}\quad
\lim_{\tau\searrow 0}\dfrac{Z_{1}(g,\tau )}{\tau }=\dfrac{c_{2}}{a_{2}}
$$
that
\begin{equation}\label{Z1}
\begin{cases}
Z_{1}(g, \tau )<Z_{1}(f, \tau )\quad &\mbox{if}\ A<C,\\
Z_{1}(f, \tau )<Z_{1}(g, \tau )\quad &\mbox{if}\ C<A.\\
\end{cases}
\end{equation}
for sufficiently small $\tau >0$.
Furthermore, it follows from
$$
\lim_{\tau\searrow 0}Z_{2}(f,\tau )=\dfrac{a_{1}}{b_{1}}
\quad\mbox{and}\quad
\lim_{\tau\searrow 0}Z_{2}(g,\tau )=\dfrac{a_{2}}{b_{2}}
$$
that
\begin{equation}\label{Z2}
\begin{cases}
(Z_{1}(f, \tau )<\,)\,Z_{2}(g, \tau )<Z_{2}(f, \tau )\quad &\mbox{if}\ B<A,\\
(Z_{1}(g, \tau )<\,)\,Z_{2}(f, \tau )<Z_{2}(g, \tau )\quad &\mbox{if}\ A<B.\\
\end{cases}
\end{equation}
for sufficiently small $\tau >0$.
Therefore, we know from \eqref{Z1} and \eqref{Z2} that
\begin{equation}\label{Z1Z2}
\begin{cases}
Z_{1}(g, \tau )<Z_{1}(f, \tau )<Z_{2}(g, \tau )<Z_{2}(f, \tau )
\quad&\mbox{if}\ B<A<C,\\
Z_{1}(f, \tau )<Z_{1}(g, \tau )<Z_{2}(f, \tau )<Z_{2}(g, \tau )
\quad&\mbox{if}\ C<A<B\\
\end{cases}
\end{equation}
for sufficiently small $\tau >0$.

In the strong competition case $B<A<C$,
we know from
\eqref{hdef} and \eqref{Z1Z2} that
three zeros $z_{j}(\tau )$ 
$(j=1,2,3)$
of $h(u,\tau )$ are located as
$$
0<z_{1}(\tau )<Z_{1}(g,\tau )<Z_{1}(f,\tau )<z_{2}(\tau )<Z_{2}(g,\tau )
<Z_{2}(f,\tau )<z_{3}(\tau )
$$
if $\tau >0$ is sufficiently small.
Hence the intermediate value theorem yields the assertion (i).

Next we consider the weak competition case $C<A<B$.
By virtue of \eqref{Z1Z2},
the intermediate value theorem ensures 
that the number of zeros of $h(u,\tau )$
in $u\in (Z_{1}(g,\tau), Z_{2}(f,\tau ))$
is one or three
if $\tau >0$ is sufficiently small.
Here we recall Lemma \ref{hlem} to note
\begin{equation}
\begin{split}
&\dfrac{1}{\tau}f\biggl(z_{1}(\tau ), \dfrac{\tau}{z_{1}(\tau )}\biggr)
=\dfrac{z_{1}(\tau )}{\tau}
(a_{1}-b_{1}z_{1}(\tau ))-c_{1}
\to \dfrac{c_{2}}{a_{2}}a_{1}-c_{1}=c_{2}(A-C)>0,\\
&\dfrac{1}{\sqrt{\tau}}f\biggl(z_{2}(\tau ), \dfrac{\tau}{z_{2}(\tau )}\biggr)
=
\dfrac{z_{2}(\tau )}{\sqrt{\tau}}(a_{1}-b_{1}z_{2}(\tau ))-c_{1}\sqrt{\tau}
\to \sqrt{\gamma a_{1}a_{2}}>0
\end{split}
\nonumber
\end{equation}
as $\tau\searrow 0$.
Then $f(z_{i}(\tau ), \tau /z_{i}(\tau ))>0$
for $i=1,2$ if $\tau >0$ is sufficiently small.
Hence
the number of zeros of $h(u,\tau )$
in $u\in (Z_{1}(g,\tau), Z_{2}(f,\tau ))$
is three, thereby,
$f(z_{i}(\tau ), \tau /z_{i}(\tau ))>0$
for $i=1,2,3$.
Consequently, the assertion (ii) follows.
\end{proof}

The next lemma gives infinitely many pieces of 
local bifurcation curves that bifurcate from
$(d^{(j)}, u^{*}, \tau^{*})$ for every $j\in\mathbb{N}$,
where  $d^{(j)}$ is the positive number defined by \eqref{bif}.
\begin{lem}\label{biflem}
Suppose that $B<A<C$ or $C<A<B$.
Suppose further that $D(a_{i}, b_{i}, c_{i}, \gamma )>0$.
For each $j\in\mathbb{N}$,
there exists a local curve
$$\varGamma_{j,\varepsilon}=
\{\,(d(s), u(s), \tau (s))\,:\,
-\varepsilon
<s<\varepsilon\,\}\subset\mathbb{R}_{+}\times X
$$
such that
\begin{equation}
\begin{split}
&\varGamma^{+}_{j,\varepsilon}:=
\{\,(d(s), u(s), \tau (s))\,:\,
0
<s<\varepsilon\,\}\subset
\mathcal{S}^{+}_{j},\\
&\varGamma^{-}_{j,\varepsilon}:=
\{\,(d(s), u(s), \tau (s))\,:\,
-\varepsilon
<s<0\,\}\subset
\mathcal{S}^{-}_{j},\\
&\lim\limits_{s\to 0}(d(s), u(s), \tau (s))= (d^{(j)}, u^{*}, \tau^{*})
\quad\mbox{in}\ \mathbb{R}\times X.
\end{split}
\nonumber
\end{equation}
\end{lem}

\begin{proof}
By a similar manner as the proof of 
\cite[Theorem 4.1]{Ku2},
we can construct the required local curve 
$\varGamma_{j,\varepsilon}$ which
forms a piece of a bifurcation curve of 
solutions of \eqref{IS1} bifurcating 
from the pitchfork bifurcation
point $(d^{(j)}, u^{*}, \tau^{*})$.

Actually, in view of the proof of Lemma \ref{indexlem},
one can recall that
the operator $F(d, w, \tau )$ in \eqref{Fdef}
associated with \eqref{IS1} 
is degenerate, in the sense that
the operator $I-L(d)$ 
has a zero eigenvalue,
if and only if $d=d^{(j)}$ with some $j\in\mathbb{N}$.,
Additional conditions for use of
the local bifurcation theorem
\cite[Theorem 1.7]{CR}
by Crandall and Rabinowitz
can be verified by a similar argument
to the proof of \cite[Theorem 4.1]{Ku2}.
\end{proof}
By virtue of Lemma \ref{hlem}, we set
$$
\widetilde{T}:=\sup\{\,T\,:\,
h(u,\tau)\ \mbox{has three zeros}\ 
0<z_{1}(\tau)<z_{2}(\tau)<z_{3}(\tau )\ \mbox{for any}\ \tau\in (0,T)\,\},
$$
and
$$
\widetilde{\tau}=\min\{\,\widetilde{T},\overline{\tau}\,\},
$$
where $\overline{\tau}$ is defined by 
Theorem \ref{nonexthm}.
\begin{lem}
Suppose that $B<A<C$ or $C<A<B$.
Suppose further that $D(a_{i}, b_{i}, c_{i}, \gamma )>0$.
Then it holds that $\tau^{*}\in\mathcal{T}$ and $u^{*}=z_{2}(\tau^{*})$.
Furthermore,
the following {\rm (i)} and {\rm (ii)} hold true:
\begin{enumerate}[{\rm (i)}]\label{taulem}
\item
If $B<A<C$ and $D(a_{i}, b_{i}, c_{i}, \gamma )>0$, then 
$\tau^{*}\in (0,\widetilde{\tau})$.
\item
If $C<A<B$ and $D(a_{i}, b_{i}, c_{i}, \gamma )>0$, 
then 
$\tau^{*}\not\in (0,\widetilde{\tau})$.
\end{enumerate}
\end{lem}

\begin{rem}\label{rem69}
By virtue of $\tau^{*}\in\mathcal{T}$,
the assertion (ii) of Lemma \ref{taulem} implies that
$\mathcal{T}$ is not connected in case where
$C<A<B$ and $D(a_{i}, b_{i}, c_{i}, \gamma )>0$.
An example of profiles of $h(u,\tau)$ in case
$B<A<C$ and $D(a_{i}, b_{i}, c_{i}, \gamma )>0$
is shown in Figure 2.
In the same setting as \cite[Figure 11]{BKS}
for the case $C<A<C$ and $D(a_{i}, b_{i}, c_{i}, \gamma )>0$,
profiles of $h(u,\tau)$ are shown in Figure 3.
In view of Figure 3, one can see that
$\mathcal{T}$ is not connected in this case.
\end{rem}

\begin{figure}
\centering
\subfigure[$\tau=\tau^{*}=1/9$]{
\includegraphics*[scale=0.4]{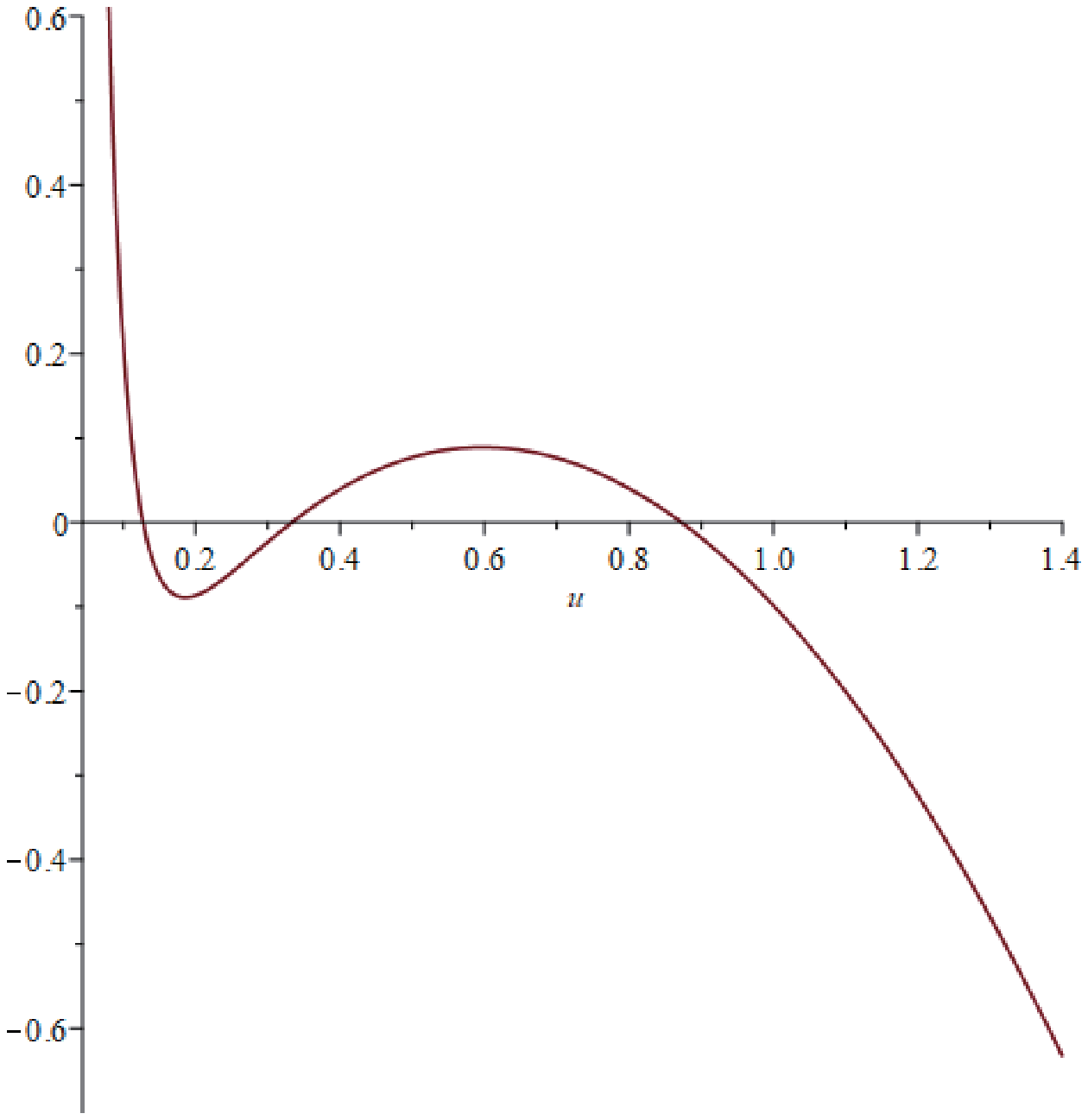}
\label{fig2a}}
\subfigure[$\tau=1/3$]{
\includegraphics*[scale=0.4]{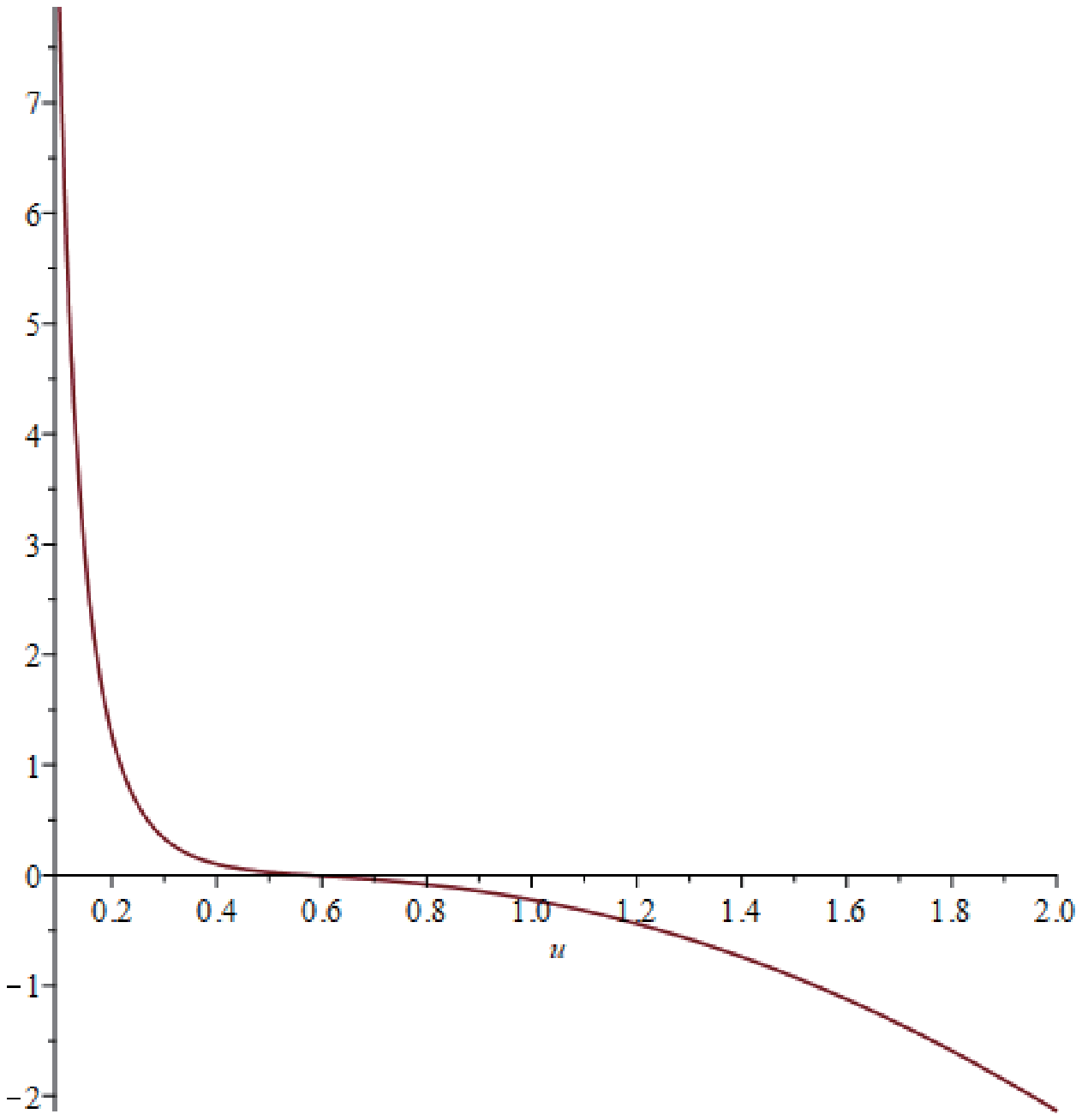}
\label{fig2b}
}

\caption{
Profiles of $h(u,\tau )$ with $(a_{1}, a_{2}, b_{1}, b_{2}, c_{1}, c_{2},\gamma)=
(1,1,1,2,2,1,1)$}
\label{fig2}
\end{figure}

\begin{figure}
\centering
\subfigure[$\tau=1/20$]{
\includegraphics*[scale=0.4]{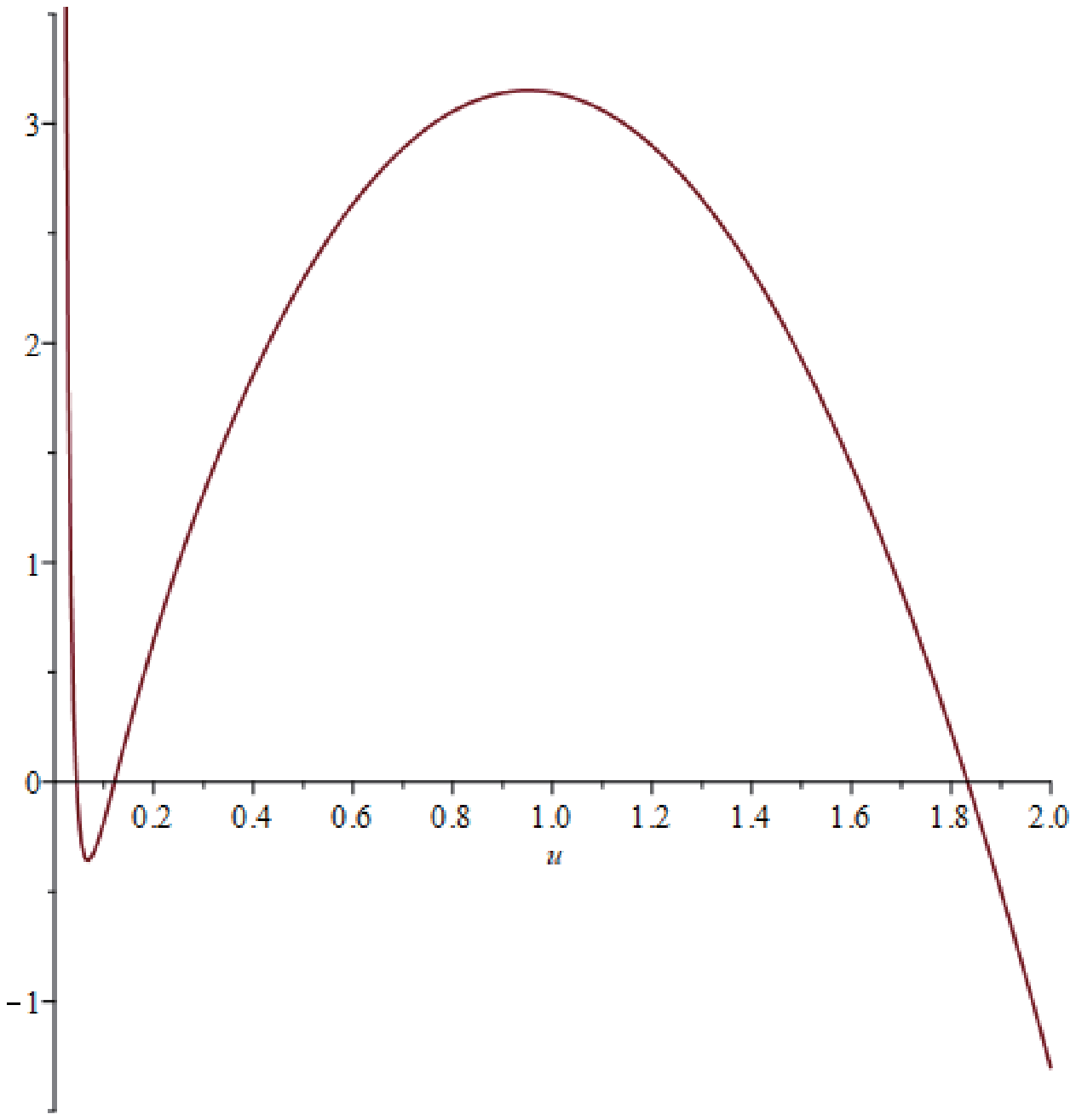}
\label{fig3a}}
\subfigure[$\tau=1/3$]{
\includegraphics*[scale=0.4]{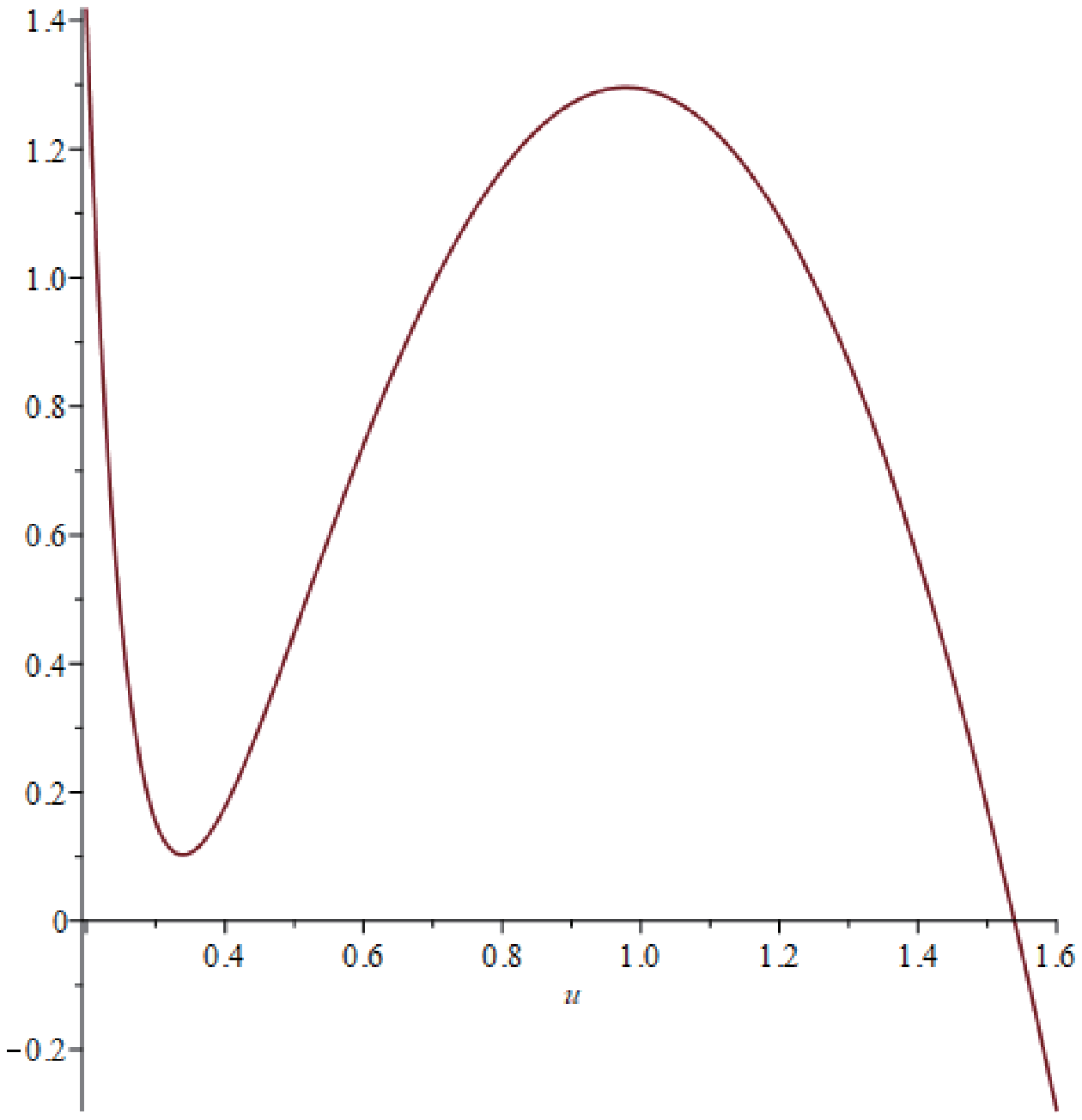}
\label{fig3b}
}

\centering
\subfigure[$\tau=\tau^{*}=207/392$]{
\includegraphics*[scale=0.4]{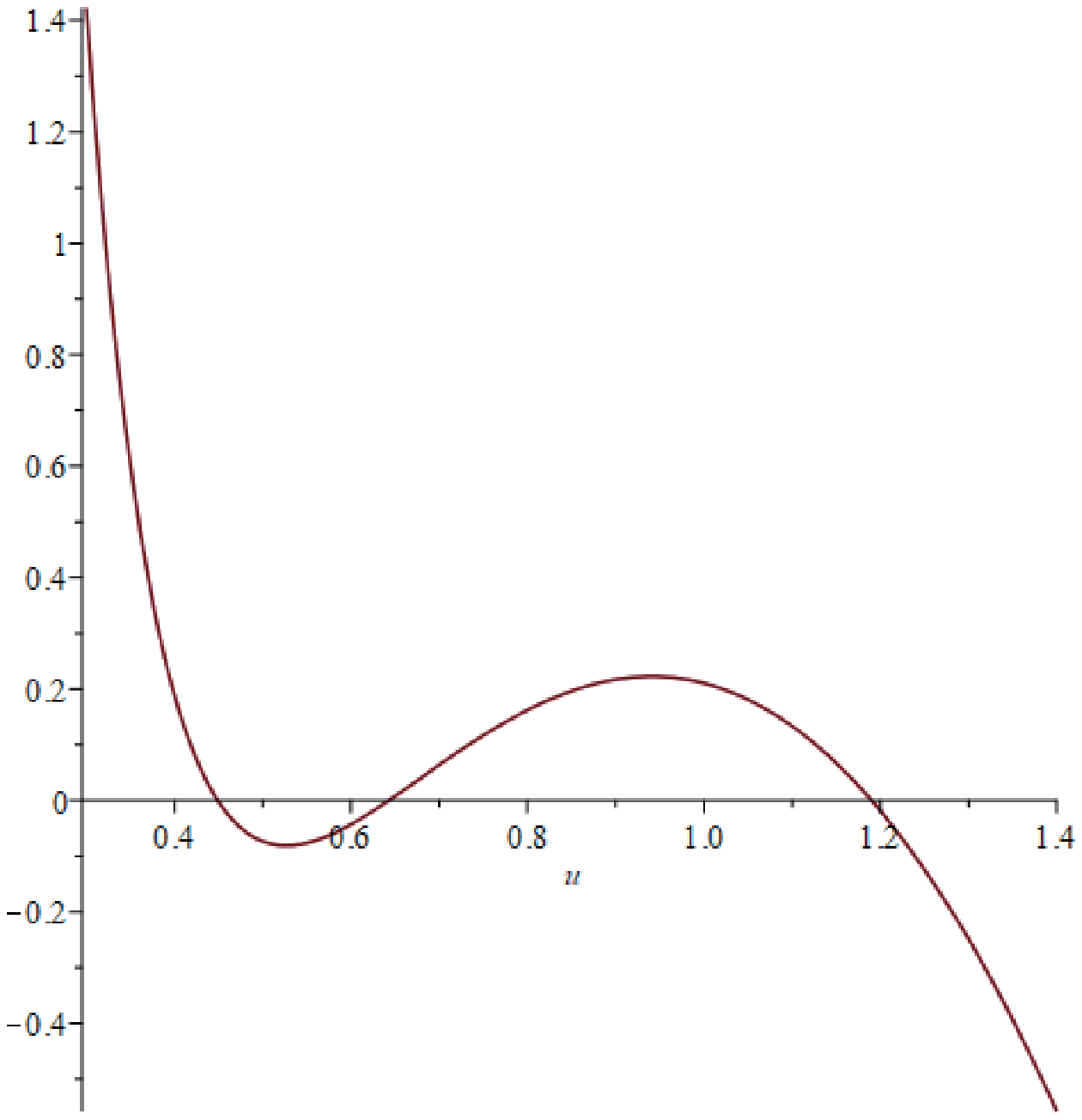}
\label{fig3c}}
\subfigure[$\tau=1$]{
\includegraphics*[scale=0.4]{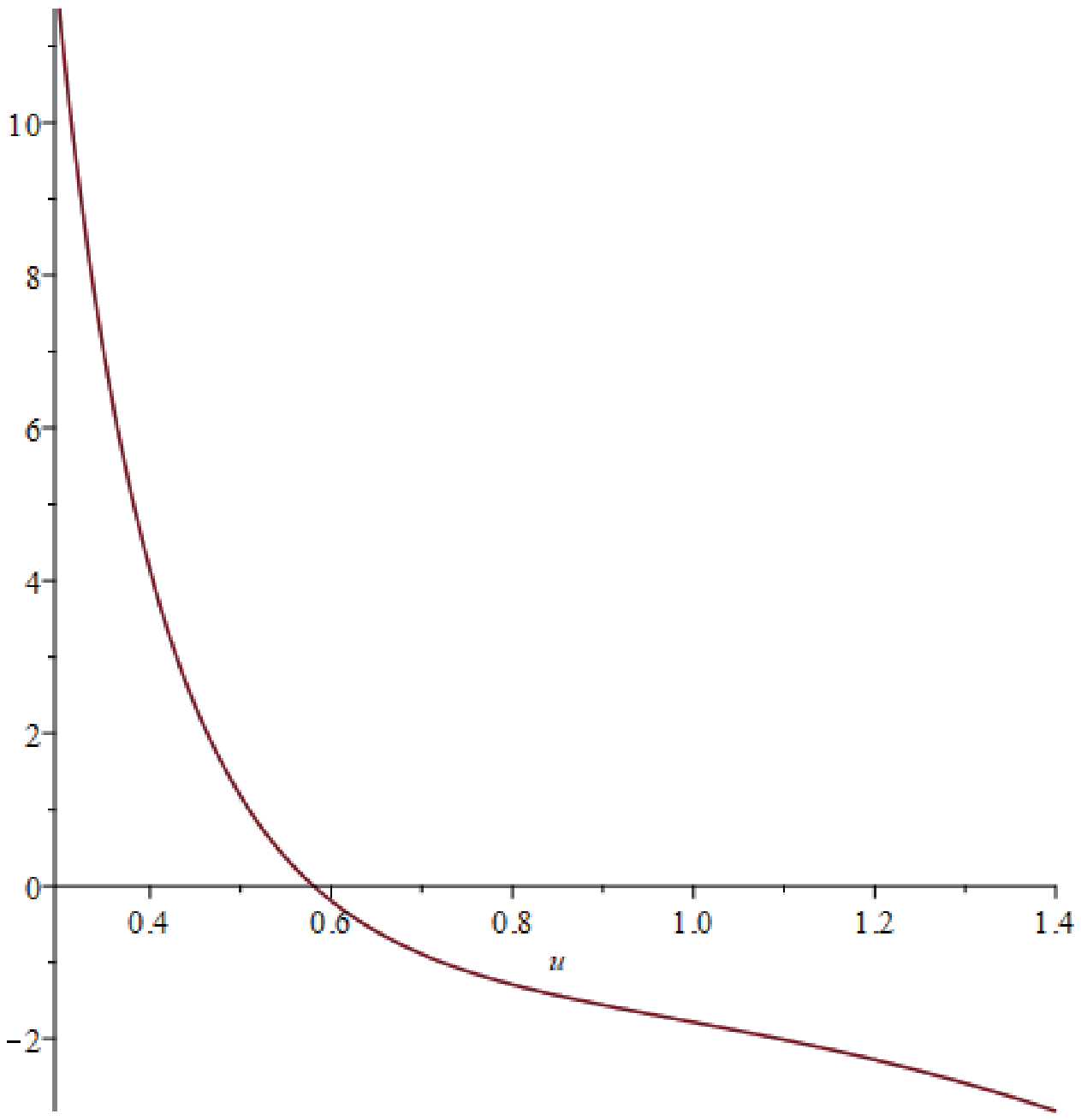}
\label{fig3d}
}

\caption{
Profiles of $h(u,\tau )$ with $(a_{1}, a_{2}, b_{1}, b_{2}, c_{1}, c_{2},\gamma)=
(15/2, 16/7, 4,1,6,2,1)$}
\label{fig3}
\end{figure}

\begin{proof}[Proof of Lemma \ref{taulem}]
Suppose that $B<A<C$ or $C<A<B$.
We recall that \eqref{IS1}
has a unique positive constant solution
$(u,\tau)=(u^{*},\tau^{*})$ 
$(w^{*}=\delta u^{*}-\gamma\tau^{*}/u^{*})$
for any $d>0$, that is,
$$
f(u^{*}, v^{*})=g(u^{*}, v^{*})=0
\quad\mbox{with}\quad v^{*}=\dfrac{\tau^{*}}{u^{*}}.
$$
By a straightforward calculation, one can verify that
$D(a_{i}, b_{i}, c_{i}, \gamma )>0$ is equivalent to 
$$
h_{u}(u^{*}, \tau^{*})>0.
$$
Taking account for facts that $\lim_{u\searrow 0}h(u,\tau)=\infty$,
$h(u^{*},\tau^{*})=0$ and $\lim_{u\to\infty}h(u,\tau )=-\infty$,
we see that
\begin{equation}\label{const2}
u^{*}=z_{2}(\tau^{*}),
\end{equation}
thereby,
\begin{equation}\label{const22}
f\biggl(z_{2}(\tau^{*}), \dfrac{\tau^{*}}{z_{2}(\tau^{*})}\biggr)
=g\biggl(z_{2}(\tau^{*}), \dfrac{\tau^{*}}{z_{2}(\tau^{*})}\biggr)=0.
\end{equation}
Furthermore, we obtain $\tau^{*}\in\mathcal{T}$.
Actually, if $\tau^{*}\not\in\mathcal{T}$, then
$\tau^{*}>\overline{\tau}$ because $h(u,\tau^{*})$ has three zeros
on $\{\,u>0\,\}$.
However, 
\eqref{tauest1} and \eqref{tauest2}
ensure $\tau^{*}\le\overline{\tau}$.
This is a contradiction.

In order to prove the assertion (i),
we assume that $B<A<C$ and $D(a_{i}, b_{i}, c_{i}, \gamma )>0$.
In view of (i) of Lemma \ref{fsflem}, we recall that
\begin{equation}\label{wcf}
f\biggl(z_{1}(\tau ), \dfrac{\tau}{z_{1}(\tau )}\biggr)<0\quad\mbox{and}\quad
f\biggl(z_{3}(\tau ), \dfrac{\tau}{z_{3}(\tau )}\biggr)<0
\end{equation}
and
\begin{equation}\label{z2p}
f\biggl(z_{2}(\tau ),\dfrac{\tau}{z_{2}(\tau )}\biggr)>0
\end{equation}
for sufficiently small $\tau >0$.
We shall show that \eqref{wcf} holds true for any 
$\tau\in (0, \widetilde{\tau})$.
Suppose for contradiction that there exists
$\tau_{0}\in (0,\widetilde{\tau})$ such that
\begin{equation}\label{13}
f\biggl(z_{i_{0}}(\tau_{0}),\dfrac{\tau_{0}}{z_{i_{0}}(\tau_{0})}\biggr)=0
\quad\mbox{with some}\ i_{0}\in\{\,1,3\,\}.
\end{equation}
Furthermore, $h(z_{i_{0}}(\tau_{0}), \tau_{0})=0$ implies
$$
f\biggl(z_{i_{0}}(\tau_{0}),\dfrac{\tau_{0}}{z_{i_{0}}(\tau_{0})}\biggr)=
g\biggl(z_{i_{0}}(\tau_{0}),\dfrac{\tau_{0}}{z_{i_{0}}(\tau_{0})}\biggr)=0.
$$
Hence it follows that $(z_{i_{0}}(\tau_{0}), \tau_{0})=(u^{*}, \tau^{*})$.
Together with \eqref{const2}, we see that
$z_{2}(\tau^{*})=z_{i_{0}}(\tau^{*})$, thereby,
the degeneracy of zeros of $h(u,\tau)$ occurs at $\tau=\tau^{*}$.
However, this contradicts the fact $\tau^{*}\in\mathcal{T}$.
Therefore,
by taking account for the uniqueness 
of positive roots of $f(u,\tau/u)=g(u,\tau/u)=0$,
we know from \eqref{const22} and \eqref{z2p} that
$$
f\biggl(z_{2}(\tau),\dfrac{\tau}{z_{2}(\tau )}\biggr)
\begin{cases}
>0\quad &\mbox{for}\ \tau\in (0, \tau^{*}),\\
=0\quad &\mbox{for}\ \tau=\tau^{*},\\
<0\quad &\mbox{for}\ \tau\in (\tau^{*},\infty)\cap\mathcal{T},
\end{cases}
$$
and moreover,
\eqref{wcf} holds true 
for any $\tau\in (0, \widetilde{\tau})$.
Together with $\tau^{*}<\overline{\tau}$,
we obtain
$\tau^{*}<\widetilde{\tau}$.

For the proof of the assertion (ii),
we assume that $C<A<B$ and $D(a_{i}, b_{i}, c_{i}, \gamma )>0$.
From (ii) of Lemma \ref{fsflem}, we recall that
$$
f\biggl(z_{1}(\tau ), \dfrac{\tau}{z_{1}(\tau )}\biggr)>0,
\quad
f\biggl(z_{2}(\tau ), \dfrac{\tau}{z_{2}(\tau )}\biggr)>0,
\quad
f\biggl(z_{3}(\tau ), \dfrac{\tau}{z_{3}(\tau )}\biggr)>0$$
for sufficiently small $\tau >0$.
Suppose for contradiction that
$\tau^{*}\in (0, \widetilde{\tau})$.
Then by \eqref{const22} and the continuity of 
$\tau\mapsto f(z_{i}(\tau ), \tau/z_{i}(\tau ))$,
there exists $\tau_{0}\in (0,\tau^{*})$ satisfying 
\eqref{13}.
As in the argument above, 
we are led to $z_{2}(\tau^{*})=z_{i_{0}}(\tau^{*})$ with some
$i_{0}\in\{1,3\}$. 
Again this contradicts the fact $\tau^{*}\in\mathcal{T}$.
Therefore, we can deduce that $\tau^{*}\not\in (0, \overline{\tau})$
in case that $C<A<B$ and $D(a_{i}, b_{i}, c_{i}, \gamma )>0$.
The proof of Lemma \ref{taulem} is complete.
\end{proof}

\begin{proof}[Proof of Theorem \ref{1dimthm}]
In what follows, we denote by $\varGamma^{\pm}_{j}$ 
the connected component of 
$
\{\,(d, u,  \tau)\in \mathbb{R}_{+}\times X\,:\,
\mbox{$(d, u, \tau)$ satisfies \eqref{IS1}}\,\}
$
which contains $\varGamma^{\pm}_{j,\varepsilon}$
obtained in Lemma \ref{biflem}.
We recall Remark \ref{taubarrem} to note that
any $(d, u,\tau)\in \varGamma^{\pm}_{j}$ satisfies 
$\tau\in\mathcal{T}$.

We first show
$\varGamma^{+}_{1}\subset\mathcal{S}^{+}_{1}$,
that is,
what $\varGamma^{+}_{1}$ does not connects with
any other $\varGamma^{\pm}_{j}$.
If not, 
there exist a solution $(\hat{d}, \hat{u}, \hat{\tau})
\,(\,\neq (d^{(1)},u^{*},\tau^{*})\,)$ 
of \eqref{IS1} and a sequence
$\{\,(d_{2,n}, u_{n}, \tau_{n})\,\}\subset\varGamma^{+}_{1}$ such that
$\lim_{n\to\infty}(d_{2,n}, u_{n}, \tau_{n})
=(\hat{d}, \hat{u}, \hat{\tau})$ in $\mathbb{R}\times X$
and $\hat{u}(x)$ has a degenerate critical point, that is,
$\hat{u}''(x_{0})=\hat{u}'(x_{0})=0$ with some $x_{0}\in [0,1]$.
Differentiating \eqref{IS1-1} by $x$, we see that
$\hat{w}:=\delta\hat{u}-\gamma\hat{\tau}/\hat{u}$
satisfies
$$
\begin{cases}
\hat{d}(\hat{w}')''+h_{u}(\hat{u},\hat{\tau})\hat{u}'=0,\quad 0<x<1,\\
\hat{w}'(x_{0})=\hat{w}''(x_{0})=0,
\end{cases}
$$
where the initial condition at $x_{0}$ comes from \eqref{wp}.
By the uniqueness of solutions of this initial value problem,
one can see that $\hat{w}'=0$, thereby,
$\hat{u}$ is a nonnegative constant.
If $\hat{\tau}>0$, then $(\hat{u},\hat{\tau})=(u^{*},\tau^{*})$.
In view of Lemma \ref{biflem}, we recall that
the branch of monotone solutions of \eqref{IS1} bifurcates from
the constant solution $(u^{*},\tau^{*})$ only at $d=d^{(1)}$.
This fact implies $\hat{d}=d^{(1)}$.
However, it contradicts the assumption.
If $\hat{\tau}=0$, then \eqref{IS1p} ensures that
$w_{n}=\delta u_{n}-\gamma \tau_{n}/u_{n}\to \hat{w}$
in $C^{1}(\overline{\Omega})$ and $(\hat{d},\hat{w})$ satisfies
\eqref{CS}.
Then, Proposition \ref{CSprop} leads to
$\hat{w}=\delta a_{1}/b_{1}$ or $\hat{w}=0$ or $\hat{w}=-\gamma a_{2}/c_{2}$.
However, one can verify that all of them are impossible 
following the argument below \eqref{below} in the proof of Lemma \ref{bddlem}.
Consequently, $\hat{\tau}=0$ is also impossible.
Therefore, we obtain $\varGamma^{+}_{1}\subset\mathcal{S}^{+}_{1}$.
By an essentially same argument, one can verify that
$\varGamma_{j}^{+}\subset\mathcal{S}_{j}^{+}$ and
$\varGamma_{j}^{-}\subset\mathcal{S}_{j}^{-}$ 
for each $j\in\mathbb{N}$.

According to the global bifurcation theorem
\cite{Ra}
by Rabinowitz
(see also
the unilateral global bifurcation theorem by 
Lop\'ez-Gom\'ez \cite[Theorem 6.4.3]{Lo}),
we can deduce that
$\varGamma^{+}_{1}$ 
reaches a singular limit $d\searrow 0$ or
a state with some $d>0$ and $\tau =0$ 
because $\varGamma^{+}_{1}$ cannot attain any 
bifurcation point $(d^{(j)}, u^{*}, \tau^{*})$,
the $(d,\tau )$ component of $\varGamma^{+}_{1}$
is uniformly bounded by Theorem \ref{nonexthm},
and the $u$ component of $\varGamma^{+}_{1}$ cannot blow up in 
$C^{1}([0,1])$ at any positive $d>0$ by Corollary \ref{Mcor}.
Then,
we take any sequence
$\{(u_{n}, \tau_{n}, d_{2,n})\}\subset\varGamma^{+}_{1}$
satisfying 
$\lim_{n\to\infty}d_{2,n}\tau_{n}=0$.
It follows from Theorem \ref{nonexthm} 
that
$\tau_{n}\in (0,\overline{\tau})$ for any $n\in\mathbb{N}$.
Then we may assume that
$\lim_{n\to\infty}\tau_{n}=\tau_{0}$
for some $\tau_{0}\in [0,\overline{\tau}]$
by passing to a subsequence if necessary.
We shall show that either of the following is true:
\begin{enumerate}[(I)]
\item
$\tau_{0}=0$;
\item
$\tau_{0}>0$
and 
$H(z_{1}(\tau_{0}), \tau_{0})=
H(z_{3}(\tau_{0}), \tau_{0})$.
\end{enumerate}
Suppose for contradiction 
that 
$\tau_{0}>0$ and
$H(z_{1}(\tau_{0}), \tau_{0})<
H(z_{3}(\tau_{0}), \tau_{0})$.
Hence it follows that $\lim_{n\to\infty}d_{2,n}=0$.
Then we can verify
\begin{equation}
\lim_{n\to\infty}u_{n}(x)=
\begin{cases}
z_{1}(\tau_{0} )\quad &\mbox{for}\ x\in [0,1),\\
\eta (\tau_{0} ) &\mbox{for}\ x=1
\end{cases}
\nonumber
\end{equation}
by a slight modification of the usual scaling procedure 
to prove (i) of Lemma \ref{singlem}.
Here we remark that
$\{u_{n}\}$ satisfies the integral constraint \eqref{IS1-3}
as well as \eqref{gint2}:
\begin{equation}\label{intk}
\displaystyle\int^{1}_{0}
f\biggl(u_{n}, \dfrac{\tau_{n}}{u_{n}}\biggr)=
\displaystyle\int^{1}_{0}
g\biggl(u_{n}, \dfrac{\tau_{n}}{u_{n}}\biggr)=
0
\quad\mbox{for any}\ n\in\mathbb{N}.
\end{equation}
By the Lebesgue dominated convergence theorem,
we set $n\to\infty$ in \eqref{intk} to get
$$
f\biggl(z_{1}(\tau_{0}), \dfrac{\tau_{0}}{z_{1}(\tau_{0})}
\biggr)=
g\biggl(z_{1}(\tau_{0}), \dfrac{\tau_{0}}{z_{1}(\tau_{0})}
\biggr)=
0.
$$
Therefore, the positivity of $\tau_{0}$ leads to
$(z_{1}(\tau_{0}), \tau_{0})=(u^{*}, \tau^{*})$.
Together with Lemma \ref{taulem}, we see that
$z_{1}(\tau^{*})=u^{*}=z_{2}(\tau^{*})$.
Obviously, this contradicts the fact that $\tau^{*}\in\mathcal{T}$.

Suppose for contradiction that
$\tau_{0}>0$ and
$H(z_{1}(\tau_{0}), \tau_{0}))
>H(z_{3}(\tau_{0} ), \tau_{0})$.
Then we know from (iii) of Lemma \ref{singlem} that
\begin{equation}
\lim_{k\to\infty}u_{k}(x)=
\begin{cases}
\zeta (\tau_{0} ) &\mbox{for}\ x=0,\\
z_{3}(\tau_{0} )\quad &\mbox{for}\ x\in (0,1].
\end{cases}
\nonumber
\end{equation}
In a similar manner,
we set $n\to\infty$ in \eqref{intk} to see
$z_{3}(\tau^{*})=u^{*}=z_{2}(\tau^{*})$.
However, it is impossible because $\tau^{*}\in\mathcal{T}$.
Consequently, the contradiction argument enables us to 
conclude that (I) or (II) holds true.

Next we shall show $\tau_{0}=0$ in the case
when $B<A<C$ and $D(a_{i}, b_{i}, c_{i}, \gamma )>0$.
It follows from (i) of Lemma \ref{taulem} that
$\tau^{*}\in(0, \widetilde{\tau})$.
By virtue of Lemma \ref{biflem},
if $(d,u,\tau )\in\varGamma^{\pm}_{j}$ is sufficiently close to
the bifurcation point $(d^{(j)}, u^{*}, \tau^{*})$, then
$\tau\in (0,\widetilde{\tau})$.
By taking account for the continuity of $\varGamma^{\pm}_{j}$, 
we see that $\tau\in\mathcal{T}$ and 
$\tau\in (0, \widetilde{\tau})$
as long as $(d, u, \tau )\in\varGamma_{j}^{\pm}$.
Hence it follows that
$\tau_{0}\in [0,\widetilde{\tau}]$.
Suppose for contradiction that
$\tau_{0}>0$.
Then, as we have already shown, 
$H(z_{1}(\tau_{0}), \tau_{0})=H(z_{3}(\tau_{0}), \tau_{0})$
follows. 
In such a case,
we remark that a slight delicate procedure
is required to derive the singular limiting behavior
of $\{u_{n}\}$ as $n\to\infty$.
Following the argument in the proof of 
\cite[Proposition 6.7]{KT} for instance, we can deduce that
\begin{equation}
\lim_{n\to\infty}u_{n}(x)=
\begin{cases}
z_{1}(\tau )\quad &\mbox{for}\ x\in [0,\ell),\\
(z_{1}(\tau )+z_{3}(\tau ))/2\quad &\mbox{for}\ x=\ell,\\
z_{3}(\tau ) &\mbox{for}\ x\in (\ell, 1]
\end{cases}
\nonumber
\end{equation}
with some $\ell\in (0,1)$.
Owing to the Lebesgue dominated convergence theorem,
we set $n\to\infty$ in 
$\int^{1}_{0}f(u_{n}, \tau_{n}/u_{n})=0$ to obtain
\begin{equation}\label{bal}
\ell
f\biggl(z_{1}(\tau_{0}), \dfrac{\tau_{0}}{z_{1}(\tau_{0})}
\biggr)
+(1-\ell)
f\biggl(z_{3}(\tau_{0}), \dfrac{\tau_{0}}{z_{3}(\tau_{0})}
\biggr)=0,\\
\end{equation}
If $B<A<C$ and $D(a_{i}, b_{i}, c_{i}, \gamma )>0$,
\eqref{bal} is impossible because
$f(z_{1}(\tau_{0}), \tau_{0}/z_{1}(\tau_{0}))<0$
and
$f(z_{3}(\tau_{0}), \tau_{0}/z_{3}(\tau_{0}))<0$
as in the proof of Lemma \ref{taulem} below \eqref{z2p}.
Consequently, this contradiction 
excludes the situation (II).
Therefore,
it holds true that $\tau_{0}=0$ 
in case where
$B<A<C$ and $D(a_{i}, b_{I}, c_{i}, \gamma )>0$.
In this case, we shall show $\lim_{n\to\infty}d_{2,n}=0$.
It follows from Theorem \ref{nonexthm} that
$d_{2,n}\in (0, \overline{d}\,]$ for any $n\in\mathbb{N}$.
Therefore, we may assume that
$\lim_{n\to\infty}d_{2,n}=d_{0}$ with some
$d_{0}\in [\,0,\overline{d}\,]$ 
by passing to a subsequence if necessary.
From Lemma \ref{aprlem} and \eqref{u+def},
we see that $\|u_{n}\|_{\infty}$ is uniformly bounded
with respect to $n\in\mathbb{N}$.
Then by \eqref{IS-1},
$w_{n}:=\delta u_{n}-\gamma \tau_{n}/u_{n}$
satisfies
$\lim_{n\to\infty}w_{n}=w_{0}$
with some $w_{0}\in C^{1}(\overline{\Omega })$.
It follows from Proposition \ref{CSprop} that
if $d_{0}>0$,
then $w_{0}=\delta a_{1}/b_{1}$ or
$w_{0}=0$ or
$w_{0}=-\gamma a_{2}/c_{2}$ in $\Omega$.
However, repeating the argument below \eqref{below} in 
the proof of Lemma \ref{bddlem},
we see that the above three situations cannot occur.
Then, we can deduce that $d_{0}=0$.
Consequently, we know that
$\lim_{n\to\infty}(d_{2,n}, \tau_{n})=(0,0)$
in case where $B<A<C$ and $D(a_{i}, b_{i}, c_{i}, \gamma )>0$.
Obviously, the uniqueness of limits ensures that
the full sequence $\{(d_{2,n}, \tau_{n})\}$ itself converges to
$(0,0)$. 
Furthermore, one can verify that all $\varGamma_{j}^{\pm}$ 
approach $(d,\tau )=(0,0)$ in a similar manner.

Next we consider the case when $C<A<B$ and $D(a_{i}, b_{i}, c_{i}, \gamma )>0$.
It follows from Remark \ref{rem69} that
$\mathcal{T}$ is not connected and $\tau^{*}\not\in (0,\widetilde{\tau})$.
This fact implies that
the projection of $\varGamma^{+}_{1}$ on the $\tau$ axis
does not intersect with $(0,\widetilde{\tau}\,]$.
Then the situation (I); $\tau_{0}=0$
cannot occur, and then, (II) necessarily occurs.
Consequently, we deduce that 
$\lim_{n\to\infty}(d_{2,n}, \tau_{n})=(0,\tau_{0})$
with some $\tau_{0}\in (\widetilde{\tau}, \overline{\tau})$
satisfying $H(z_{1}(\tau_{0}), \tau_{0})=
H(z_{3}(\tau_{0}), \tau_{0})$.
Hence it follows that $\varGamma_{1}^{+}$ approaches 
$(d, \tau)=(0,\tau_{0})$.
It is possible to check that each
$\varGamma_{j}^{\pm}$ also attains $(d, \tau)=(0, \tau_{0})$.
The proof of Theorem \ref{1dimthm} is complete. 
\end{proof}

%

%
%

\end{document}